\documentclass[10pt,reqno,centertags]{amsart}

\usepackage{amssymb,amsmath,amsthm} \usepackage{dsfont}

\newcommand{\e}{\ensuremath{\mathbf{e}}}
\newcommand{\R}{\mathbb{R}} \newcommand{\Z}{\mathbb{Z}}
 
\newcommand{\N}{\mathbb{N}} \newcommand{\s}{\mathbb{S}}

\theoremstyle{plain} \newtheorem{theorem}{Theorem}[section]
\newtheorem{lemma}[theorem]{Lemma}
\newtheorem{corollary}[theorem]{Corollary}

\theoremstyle{definition}
 \theoremstyle{remark}

\begin{document}

\title[Global well-posedness for Schr\"{o}dinger equation]{%
  Global well-posedness and scattering for Derivative Schr\"{o}dinger equation}
\author{Baoxiang Wang, \,  Yuzhao Wang}
\address{LMAM, School of Mathematical Sciences, Peking University, Beijing
100871, China}

\email{wbx@math.pku.edu.cn, wangyuzhao2008@gmail.com}

\begin{abstract}
In this paper we  study the Cauchy problem for the elliptic and non-elliptic derivative
nonlinear Schr\"odinger equations in higher spatial dimensions ($n\geq 2$) and some global
well-posedness results with small initial data in critical Besov spaces $B^s_{2,1}$ are obtained.
As by-products, the scattering results with small initial data are also obtained.\\

\end{abstract}


\maketitle

\section{Introduction, main results and notations}\label{sect:intro_main}

\subsection{Introduction}

In this paper, we mainly consider the Cauchy problem for the elliptic and non-elliptic
derivative nonlinear  Schr\"odinger (DNLS) equation
\begin{equation}\label{ds}
 (i\partial_t +\Delta_\pm) u
=F(u, \bar{u}, \nabla u, \nabla \bar{u}),
\quad
u(0,x)=u_{0}(x),
 \end{equation}
where $u$ is a complex-valued function of $(t,x)\in \mathbb{R}
\times \mathbb{R}^n$,
\begin{align}
\Delta_\pm u = \sum^n_{i=1} \varepsilon_i \partial^2_{x_i}u, \quad
\varepsilon_i \in \{1,\,  -1\}, \quad i=1,...,n,
 \label{pm}
\end{align}
$\nabla =(\partial_{x_1},..., \partial_{x_n})$ and $F:
\mathbb{C}^{2n+2} \to \mathbb{C}$ is a polynomial,
\begin{align}\label{poly}
F(z) = P(z_1,..., z_{2n+2})= \sum_{m\le |\beta|< \infty} c_\beta
z^\beta, \quad c_\beta \in \mathbb{C},
\end{align}
here $m \in \mathbb{N}$, $m\ge 3$, $n\ge 2$, $\beta=(\beta_1,
\ldots, \beta_{2n+2})\in \mathbb{Z}_+^{2n+2}$.
The DNLS covers the following derivative nonlinear Schr\"odinger equations as special cases.
\begin{align}\label{dnls-1}
 (i\partial_t +\Delta_\pm) u
=|u|^2 \vec{\lambda} \cdot\nabla u+ u^2 \vec{\mu} \cdot\nabla \bar{u},
\end{align}
\begin{align}\label{dnls-2}
 (i\partial_t +\Delta_\pm) u
=\frac{2\bar{u}}{1+|u|^2}\sum_{j=1}^n \varepsilon_j (\partial_{x_j}u)^2.
\end{align}
Eq. \eqref{dnls-1} including the non-elliptic case
describes the strongly interacting many-body systems near
criticality as recently described in terms of nonlinear dynamics
\cite{TD, DT, CT}. Eq. \eqref{dnls-2} is an equivalent form of the Schr\"odinger map
(elliptic case) and the Heisenberg map  (non-elliptic case)
\begin{align}\label{dnls-3}
  \partial_t M = M \times \Delta_\pm M
\end{align}
under the  stereographic projection (cf. \cite{BIKT, DW98, IK, IK2,
Zhou-Guo-Tan}), respectively.

The local and global well posedness of DNLS \eqref{ds} have been
extensively studied, see Bejenaru and Tataru \cite{Be-Ta}, Chihara
\cite{Chih1,Chih2}, Kenig, Ponce and Vega \cite{KPV1,KPV2},
Klainerman \cite{Klai}, Klainerman and Ponce \cite{Kl-Po}, Ozawa and
Zhai \cite{Oz-Zh}, Shatah \cite{Shat} and the authors \cite{WW}.
When the nonlinear term $F$ satisfies some energy structure
conditions, or the initial data suitably decay, the energy method,
which went back to the work of Klainerman \cite{Klai} and was
developed in \cite{Chih1,Chih2,Kl-Po,Oz-Zh, Shat}, yields the global
existence of DNLS \eqref{ds} in the elliptical case $\Delta_\pm=
\Delta$. Recently, Ozawa and Zhai obtained the global well posedness
in $H^{s}(\mathbb{R}^n)$ ($n\ge 3$, $s>2+n/2$, $m\ge 3$) with small
data for DNLS \eqref{ds} in the elliptical case, where an energy
structure condition on $F$ is still required.

By setting up the local smooth effects for the solutions of the
linear Schr\"odinger equation, Kenig, Ponce and Vega
\cite{KPV1,KPV2} were able to deal with the non-elliptical case and
they established the local well posedness of Eq. \eqref{ds} in $H^s$
with $s\gg n/2$.  Recently, the local well posedness results have
been generalized to the quasi-linear (ultrahyperbolic) Schr\"odinger
equations, see \cite{KePoVe3,KePoRoVe}. By using Kenig, Ponce and Vega's local smooth effects
\cite{KPV1} and establishing time-global maximal function estimates
in space-local Lebesgue spaces, the authors \cite{WW} also showed
the global well posedness of elliptic and non-elliptic DNLS for small data in Besov
spaces $B^s_{2,1}(\mathbb{R}^n)$ with $s>n/2+3/2$, $m\ge 3+4/n$. Wang, Han and Huang \cite{whh} was able to deal with the case $m\ge 3$ and $n\ge 3$ by using the frequency-uniform decomposition techniques, where the initial data can be in modulation spaces $M^{3/2}_{2,1}$ and so, in Sobolev space $H^s$ with $s>n/2+3/2$. However, for the initial data in critical Sobolev spaces, the global well posedness of DNLS \eqref{ds} for both elliptic and non-elliptic cases is still unsolved.

In this paper, we will improve the  results in higher spatial dimensions
 \cite{WW, whh} to critical Besov spaces.

\medskip

\subsection{Notations}
Throughout this paper, we fix $k\in \mathbb{N}$. For $x, y\in \R^+$,
$x\lesssim y$ means that there exists $C>0$ such that $x\leq Cy$. By
$x\sim y$ we mean $x\lesssim y$ and $y\lesssim x$. Let $\chi \in
C^\infty_0((-2,2))$ be an even, non-negative function such that
$\chi(s)=1$ for $|s|\leq 1$. We define $\psi(\xi):=\chi(|\xi|)-\chi(2|\xi|)$
and $\psi_j:=\psi(2^{-j}\cdot)$.  Then,
\[\sum_{j\in\Z}\psi_j(\xi)=1, \quad{ \rm for } \
\xi\in \R^n, \text{ }\xi\neq 0.\]
Define
\begin{equation*}
  \widehat{P_j u}(\xi):=\psi_j(\xi) \widehat{u}(\xi),
\end{equation*}
and $P_{\geq M}=\sum_{j\geq
  M} P_j$ as well as $P_{<M}=I-P_{\geq M}$. Also we define the operator $\tilde{P}_j$ by
  \[\tilde{P}_j= P_{j-1}+P_j+P_{j+1},\]
  which satisfies
  \[\tilde{P}_j\circ P_j=P_j.\]
We denote by $\mathcal{S}(\R^n)$ and $\mathcal{S}'(\R^n)$ the Schwartz space and its dual space, respectively. The Besov spaces $\dot B^{s}_{2,1}(\R^n)$ and $ B^{s}_{2,1}(\R^n)$ are respectively the completions of $\mathcal{S}(\R^n)$ in $\mathcal{S}'(\R^n)$ with respect to the norms
\begin{align*}
  \|u\|_{\dot B^{s}_{2,1}(\R^n)}&:= \sum_{j=-\infty}^
    {\infty}2^{sj}\|P_{j}u\|_{L^{2}};\\
  \|u\|_{B^{s}_{2,1}(\R^n)}&:=\|P_{\leq 0}u\|_{L^{2}}+ \sum_{j=1}^
    {\infty}2^{sj}\|P_{j}u\|_{L^{2}}.
\end{align*}
For Banach spaces $X$ and $Y$, we define the the Banach space $X\cap
Y$ by the norm
\begin{align*}
\|u\|_{X\cap Y}&:= \|u\|_X+\|u\|_Y
\end{align*}
and $X\cup
Y$ by the norm
\begin{align*}
\|u\|_{X\cup Y}&:= \inf\{\|f\|_X+\|g\|_Y: \,u=f+g, \,f\in X, g\in
Y\}.
\end{align*}
The Fourier transform for any Schwartz function $f $ is defined by
\begin{equation*}
  \widehat{f}(\xi)=\mathcal{F}f(\xi)=c_0
  \int_{\R^n}e^{-ix \cdot \xi}f(x) \,dx,
\end{equation*}
 and extended to $\mathcal{S}'(\R^{n})$ by
duality. In the same way, for a function $u(t,x)$ on $\R\times\R^n$,
we define its time-space Fourier transform
\begin{equation*}
  \widehat{u}(\tau,\xi)=\mathcal{F}_{t,x}u(\tau,\xi)=c_1
  \int_{\R^n\times \R}e^{-i( t\tau +x \cdot \xi)}u(t,x) \,dtdx,
\end{equation*}

For any vector $\mathbf{e}\in\mathbb{S}^{n-1}$ let
$$P_{\mathbf{e}}=\{\xi\in\mathbb{R}^n:\xi\cdot\mathbf{e}=0\}.$$
Then for $p,q\in[1,\infty]$,
define the normed spaces
$L^{p,q}_{\mathbf{e}}=L^{p,q}_{\mathbf{e}}(\mathbb{R}\times\mathbb{R}^n)$,
\begin{equation}\label{pq}
\begin{split}
&L^{p,q}_{\mathbf{e}}=\{u\in L^2(\mathbb{R}\times \mathbb{R}^n):\|u\|_{L^{p,q}_{\mathbf{e}}}<\infty\},
\\&\text{ where \ } \|u\|_{L^{p,q}_{\mathbf{e}}}=\Big[\int_{\mathbb{R}}\Big[\int_{P_\mathbf{e}\times
\mathbb{R}}|u(t,r\mathbf{e}+v)|^q\,dvdt\Big]^{p/q}\,dr\Big]^{1/p}.
\end{split}
\end{equation}
Let $\mathbf{e}_0=(1,0,\ldots,0)$, we can fix a space rotation
matrix $A$, which depend on $\e$, such that
\begin{align} \label{A}
A\mathbf{e}=\mathbf{e}_0.
\end{align}
We have
\begin{align}\label{A1}
\|u(t,x)\|_{L^{p,q}_{\e}}=\|u(t,A^{-1}x)\|_{L^{p}_{x_1}L^{q}_{\bar{x},t}},
\end{align}
where $\bar{x}\in \R^{n-1}$ and $x=(x_1, \bar{x})$.

In view of \eqref{pm}, we denote
\begin{align}
|\xi|^2_\pm = \sum^n_{i=1} \varepsilon_i \xi_i^2, \quad
\varepsilon_i \in \{1,\,  -1\}, \quad i=1,...,n,
 \label{pm1}
\end{align}
and  by $W_\pm(t)$ the linear non-elliptic Schr\"odinger semi-group
\begin{equation*}
  \widehat{W_\pm(t)\phi}(\xi)= e^{-it|\xi|^2_\pm }\widehat{\phi}(\xi).
\end{equation*}
For $\e=(e_1,\cdots, e_n)\in \s^{n-1}$, denote
\begin{align} \label{pme}
\mathbf{e}_\pm = (\varepsilon_1 e_1,\ldots,\varepsilon_n e_n)\in \mathbb{S}^{n-1}.
\end{align}
where $\varepsilon_i \in \{1,\,  -1\}$, $i=1, \ldots, n$ are the same as in \eqref{pm1}.

Let $A$ be as in \eqref{A}.  We have
\begin{align}\label{s21}
\partial_{\xi_{1}}(|A^{-1}\xi|^2_\pm):=&\partial_t\Big|_{t=0}(|A^{-1}(\xi+t\e_0)|^2_\pm)=
\partial_t\Big|_{t=0}(|A^{-1}\xi+t\e|^2_\pm)\nonumber\\=&
\partial_t\Big|_{t=0}\big(|A^{-1}\xi|^2_\pm+t^2|\e|^2_\pm+2tA^{-1}\xi\cdot
\e_\pm\big)\\=&2(A^{-1}\xi)\cdot\mathbf{e}_\pm.\nonumber
\end{align}
This derivative relation was first observed in \cite{W} and will be
used extensively in the sequel. Define the directional derivative
along $\mathbf{e}$ by
\begin{eqnarray}\label{dd}
\widehat{D_{\e}f}(\xi)=(\xi\cdot \e)\widehat{f}(\xi).
\end{eqnarray}
Now define the projection operator $Q^{\mathbf{e}}_{k,l}$ related to
$\e$ as
\begin{eqnarray}\label{q}
\widehat{Q^{\mathbf{e}}_{k,l}f}(\xi)=\psi_{\geq k-l}(\xi\cdot
\mathbf{e})\widehat{f}(\xi),
\end{eqnarray}
which are cut-offs in frequency space along the direction $\e$.

\medskip

\subsection{Main results}

First, we consider the global well-posedness of DNLS \eqref{ds}.

\begin{theorem}\label{t3}
Assume that $m\ge 3$ for $n\ge 3$, and $m\ge 4$ for $n=2$. Suppose that $u_0 \in
 B^{n/2+1}_{2,1}(\R^n)$ and $\|u_0\|_{
B^{n/2+1}_{2,1}(\R^n)}<\delta$ for some small $\delta>0$.   Then DNLS \eqref{ds} has  a unique solution
  \begin{equation*}
    u\in  \tilde{Z}_{2,1}^{n/2+1}\subset
    C(\R;  B^{n/2+1}_{2,1}(\R^n)),
  \end{equation*}
where $\tilde{Z}_{2,1}^{n/2+1}$ is defined in \eqref{no5'}.  Moreover,  the scattering operator carries a whole neighborhood in $B^{n/2+1}_{2,1}(\R^n)$ into $B^{n/2+1}_{2,1}(\R^n)$.
\end{theorem}

By expanding $1/(1+|u|^2)$ into power series, it is easy
to see that the nonlinear term in Eq. \eqref{dnls-2} is a special case of
\eqref{poly}. The result  of Theorem \ref{t3}
contains the equivalent form of the Schr\"odinger  and Heisenberg map
equation \eqref{dnls-2} as a special case if $n\ge 3$.

\begin{corollary}\label{cor-1}
Let $n\ge 3$.  Assume that $u_0 \in
 B^{n/2+1}_{2,1}(\R^n)$ and $\|u_0\|_{
B^{n/2+1}_{2,1}(\R^n)}<\delta$ for some small $\delta>0$.   Then Eq. \eqref{dnls-2} has  a unique solution
 $ u\in  \tilde{Z}_{2,1}^{n/2+1}\subset C(\R;  B^{n/2+1}_{2,1}(\R^n)) $.  Moreover,  the scattering operator carries a whole neighborhood in $B^{n/2+1}_{2,1}(\R^n)$ into $B^{n/2+1}_{2,1}(\R^n)$.
\end{corollary}

\medskip

\medskip

\medskip

Now we consider the
initial value problem
\begin{equation}\label{dsb}
 (i\partial_t +\Delta_\pm) u
=|u|^{m-1}\vec{\lambda}_1\cdot\nabla u+ |u|^{m-3} u^2 \vec{\lambda}_2 \cdot\nabla \bar{u}, \text{ }
u(0,x)=u_{0}(x),
\end{equation}
where $m\in 2\mathbb{N}+1$,
$\vec{\lambda}_1 $ and $ \vec{\lambda}_2$ are constant vectors.  Taking $m=3$ in  \eqref{dsb}, we get Eq. \eqref{dnls-1}.   The initial value problem \eqref{dsb} is invariant
under the scaling
\begin{equation}\label{eq:scalingb}
 u_\lambda(t,x) = \lambda^{\frac{1}{m-1}} u(\lambda^2 t, \lambda x),\quad u_{\lambda}(0,x)=\lambda^{\frac{1}{m-1}} u_0(\lambda x),
\end{equation}
where $\lambda>0$. Denote
\begin{equation}
\label{cr}s'_m=\frac{n}{2}-\frac{1}{m-1},
\end{equation}
 then $\|u_\lambda(0,x)\|_{\dot B^{s'_m}_{2,1}} \sim \|u(0,x)\|_{\dot B^{s'_m}_{2,1}}$ and the equivalence is independent of $\lambda>0$. From this point of view, we say that $\dot B^{s'_m}_{2,1}$ is the critical space of Eq. \eqref{dsb}.

\begin{theorem}\label{t2}
Let $m$ be an odd integer. Assume that $m\ge 3$ for $n\ge 3$  and $m\ge 5$ for $n=2$.
There exists $\delta>0$  such that for any $u_0 \in
  \dot  B^{s'_{m}}_{2,1}(\R^n)$ with $\|u_0\|_{\dot
B^{s'_{m}}_{2,1}(\R^n)}<\delta$, Eq. \eqref{dsb} has a unique solution
  \begin{equation*}
    u\in \dot Z_{2,1}^{s'_{m}}\subset
    C(\R; \dot B^{s'_{m}}_{2,1}(\R^n)),
  \end{equation*}
  where $\dot Z_{2,1}^{s'_{m}}$ is as in \eqref{no5}.
  Moreover, the scattering operator carries a whole neighborhood in   $\dot
B^{s'_{m}}_{2,1}(\R^n)$ into   $\dot
B^{s'_{m}}_{2,1}(\R^n)$.
\end{theorem}

\medskip
Finally, we consider the
nonlinearity with full derivative terms
\begin{equation}\label{dsa}
 (i\partial_t +\Delta_\pm) u
=F(\nabla u, \nabla \bar{u}), \quad u(0,x)=u_{0}(x),
\end{equation}
where $F: \mathbb{C}^{2n}\rightarrow \mathbb{C}$ is a homogeneous
polynomial of degree $m$, for example, $F(\nabla u, \nabla \bar{u})=(\partial_{x_1} u)^{m-\ell} (\partial_{x_2} u)^{\ell}$.  It is
easy to see that equation \eqref{dsa} is invariant under the scaling
  \begin{equation}\label{eq:scaling}
    u_\lambda(t,x) = \lambda^{-1+\frac{1}{m-1}} u(\lambda^2 t, \lambda x),
    \quad u_{\lambda}(0,x)=\lambda^{-1+\frac{1}{m-1}} u_0(\lambda x),
  \end{equation}
where $\lambda>0$. Denote
   \begin{equation}
       \label{cr}s_m=\frac{n}{2}+\frac{m-2}{m-1},
   \end{equation}
then $\|u_\lambda(0,x)\|_{\dot B^{s_m}_{2,1}} \sim  \|u(0,x)\|_{\dot
B^{s_m}_{2,1}}$, thus $s_m$ is also referred to as a critical index.

\begin{theorem}\label{t1}
 Assume $m\geq 3$ for $n\ge 3$, and $m\ge 4$ for $n=2$. There exists $\delta>0$, such that for any $u_0 \in
  \dot B^{s_{m}}_{2,1}(\R^n)$ with $\|u_0\|_{\dot
     B^{s_{m}}_{2,1}(\R^n)}<\delta$, Eq. \eqref{dsa} has a unique solution
  \begin{equation*}
    u\in \dot Z_{2,1}^{s_{m}}\subset
    C(\R; \dot B^{s_{m}}_{2,1}(\R^n)),
  \end{equation*}
where $\dot Z_{2,1}^{s_{m}}$ is as in \eqref{no5}.  Moreover,   the scattering operator carries a whole neighborhood in   $\dot
B^{s_{m}}_{2,1}(\R^n)$ into   $\dot
B^{s_{m}}_{2,1}(\R^n)$.
\end{theorem}

\medskip

\medskip

\medskip

The rest of the paper is organized as follows. In Section
\ref{sect:lin esti}, we deduce the $L^{p,\infty}_{\mathbf{e}}$,
$L^{\infty,2}_{\mathbf{e}}$ type estimates for the solutions of the linear Schr\"odinger equation. In section \ref{sect:homo}, we construct the
resolution spaces and prove some nonlinear estimates to deal with IVP \eqref{dsb} and \eqref{dsa}. In Section \ref{sect:gener}, we consider the IVP \eqref{ds} with general nonlinearity. In section \ref{sect:proof_main}, we prove the main
results. In the last
section, we give the rotated Christ-Kiselev Lemma
for anisotropic Lebesgue spaces.

\medskip

\section{Linear estimates}\label{sect:lin esti}

Recalling that in order to solve the Schr\"odinger map, Ionescu and Kenig \cite{IK} used the following type estimate
 \begin{align}\label{se}
\left\| P_j Q_{j,10}^{\e} e^{it\Delta}\phi \right\|_{L^{\infty,2}_\e} \lesssim 2^{-j/2}\|\phi\|_{2}£¬
 \end{align}
which is actually implied by
\begin{align}\label{ic}
\left\| D^{1/2}_{x_1} e^{it\Delta} \phi   \right\|_{L^\infty_{x_1}
L^2_{\bar{x}}L^2_t } \lesssim \|\phi\|_{2}
\end{align}
where $x=(x_1,\bar{x})$, which was first used by Linares and Ponce
\cite{Li-Po93} to study the local well-posedness of the
Davey-Stewartson system. Indeed, after a spatial rotation,
\eqref{ic} implies that \eqref{se} holds. Even though \eqref{ic}
also holds for the non-elliptic case and it is a straightforward
consequence of the local smoothing effect in one spatial dimension
(cf. \cite{KPV1, Li-Po93, whh}), \eqref{se} in the non-elliptic case
is not true, since $e^{it\Delta_\pm}$ is not invariant under the
spatial rotation. In this paper, we show the following smoothing
effect estimates  by partially using the idea of Kenig, Ponce and
Vega \cite{KPV1} in one spatial dimension:
\begin{align}
\left\|D_{\e_\pm}^{1/2}e^{it\Delta_\pm}\phi\right\|_{L^{\infty,2}_{\mathbf{e}}}&\leq
C\|\phi\|_{L^{2}},\label{ia}\\
\Big\|D_{\e_\pm} \int_0^te^{i(t-s)\Delta_\pm}P_{j}F(s)\Big\|_{L^{\infty,2}_{\e}}&\leq
C\|F\|_{L^{1,2}_\e}\label{ib},
\end{align}
where $\e_{\pm}$ is defined in \eqref{pme}, $D^{1/2}_{\e}$ and $D_{\e}$ are  defined
by the symbols $|\xi\cdot \e|^{1/2}$ and $\xi\cdot \e$, respectively.

\subsection{Smoothing estimate and Maximal Function estimate}

In this subsection we shall prove the
smoothing estimate and maximal function estimate, both of which are sharp up to scaling and global in time.

\begin{lemma}[Smoothing effect]\label{l3}
For $\phi\in L^{2}(\mathbb{R}^{n})$ and $n\geq 2$, then
\begin{eqnarray}\label{l31}
\left\|D_{\e_\pm}^{1/2}W_\pm(t)\phi\right\|_{L^{\infty,2}_{\mathbf{e}}}\leq
C\|\phi\|_{L^{2}},
\end{eqnarray}
where $D_{\e_\pm}^{1/2}$ is defined by Fourier multiplier $|\xi\cdot
\e_\pm|^{1/2}$, and $\e_\pm\in\s^{n-1}$ is defined in \eqref{pme}.
\end{lemma}
\begin{proof}[{Proof of Lemma \ref{l3}}] By the definition, we have
\begin{align}\label{l32}
D_{\e_\pm}^{1/2}W_\pm(t)\phi=&\int_{\mathbb{R}^n}e^{ix\cdot
\xi}e^{it|\xi|^{2}_\pm}|\xi\cdot
\e_\pm|^{1/2}\widehat{\phi}(\xi)\,d\xi\nonumber\\
:=&\int_{\mathbb{R}^n}e^{ix\cdot
\xi}e^{it\Psi(\xi)}a(\xi)\widehat{\phi}(\xi)\,d\xi
\end{align}
where we denote $\Psi(\xi)=|\xi|^{2}_\pm$, and $a(\xi)=|\xi\cdot
\e_\pm|^{1/2}$. In view of \eqref{l32}, \eqref{A1} and $A^{-1}=A^t$, we have
\begin{align}\label{l33}
\Big\|D_{\e_\pm}^{1/2}W_\pm(t)\phi(x)\Big\|_{L^{\infty,2}_{\mathbf{e}}}
=&\Big\|\int_{\mathbb{R}^n}e^{ix\cdot
\xi}e^{it\Psi(\xi)}a(\xi)\widehat{\phi}(\xi)\,d\xi\Big\|_{L^{\infty,2}_{\mathbf{e}}}\nonumber\\
=&\Big\|\int_{\mathbb{R}^n}e^{i(A^{-1}x)\cdot
\xi}e^{it\Psi(\xi)}a(\xi)\widehat{\phi}(\xi)\,d\xi\Big\|_{L^{\infty}_{x_1}L^{2}_{\bar{x},t}}\nonumber\\
=&\Big\|\int_{\mathbb{R}^n}e^{ix\cdot
A\xi}e^{it\Psi(\xi)}a(\xi)\widehat{\phi}(\xi)\,d\xi\Big\|_{L^{\infty}_{x_1}L^{2}_{\bar{x},t}}\nonumber\\
=&\Big\|\int_{\mathbb{R}^n}e^{ix\cdot
\xi}e^{it\Psi_1(\xi)}a_1(\xi)\widehat{\phi}_1(\xi)\,d\xi\Big\|_{L^{\infty}_{x_1}L^{2}_{\bar{x},t}}
\end{align}
where $\Psi_1(\xi)=\Psi(A^{-1}\xi), a_1(\xi)=a(A^{-1}\xi)$ and
$\widehat{\phi}_1(\xi)=\widehat{\phi}(A^{-1}\xi)$, then apply
Plancherel theorem to \eqref{l33} in $\bar{x}$ variables and continue with
\begin{align}\label{l34}
=&\Big\|\int_{\mathbb{R}}e^{ix_1
\xi_1}e^{it\Psi_1(\xi)}a_1(\xi)\widehat{\phi}_1(\xi)\,d\xi_1\Big\|_{L^{\infty}_{x_1}L^{2}_{t}L^{2}_{\bar{\xi}}}\nonumber\\
\leq&\Big\|\int_{\mathbb{R}}e^{ix_1
\xi_1}e^{it\Psi_1(\xi)}a_1(\xi)\widehat{\phi}_1(\xi)\,d\xi_1\Big\|_{L^{2}_{\bar{\xi}}L^{\infty}_{x_1}L^{2}_{t}}.
\end{align}
Then for \eqref{l31}, it is sufficient to show for any $\bar{\xi}\in \R^{n-1}$,
\begin{eqnarray}\label{l34}
\Big\|\int_{\mathbb{R}}e^{ix_1
\xi_1}e^{it\Psi_1(\xi)}a_1(\xi)\widehat{\phi}_1(\xi)\,d\xi_1\Big\|_{L^{\infty}_{x_1}L^{2}_{t}}\leq C
\|\widehat{\phi}_1(\xi)\|_{L^2_{\xi_1}} .
\end{eqnarray}
Fix $\bar{\xi}$ to be constant, performing the change of variable
$\eta=\Psi_1(\xi_1,\bar{\xi})$, using Plancherel's identity in the
$t$-variable, then returning to the original variable
$\xi_1=\theta(\eta)$, it follows that
\begin{align}\label{l35}
&\Big\|\int_{\mathbb{R}}e^{ix_1
\xi_1}e^{it\Psi_1(\xi)}a_1(\xi)\widehat{\phi}_1(\xi)\,d\xi_1\Big\|^2_{L^{2}_{t}}\nonumber\\
=&\int\Big|\int_{\mathbb{R}}e^{ix_1
\theta(\eta)}e^{it\eta}a_1(\xi)\widehat{\phi}_1(\xi)\theta'(\eta)\,d\eta\Big|^2dt\nonumber\\
=&\int\Big|a_1(\xi)\widehat{\phi}_1(\xi)\theta'(\eta)\Big|^2\,d\eta=
\int\Big|a_1(\xi)\widehat{\phi}_1(\xi)|\theta'(\eta)|^{1/2}\Big|^2\,d\xi_1,
\end{align}
now it suffices to show $a_1(\xi_1,\bar{\xi})|\theta'(\eta)|^{\frac{1}{2}}=1$, which is
equivalent to
$|\partial_{\xi_1}\Psi_1(\xi_1,\bar{\xi})|^{1/2}=a_1(\xi_1,\bar{\xi})$, by
the definition of $\Psi_1$, it is sufficient to show
\begin{eqnarray}\label{l36}
|\partial_{\xi_1}\Psi(A^{-1}\xi)|=2(a_1(\xi))^2
\end{eqnarray}
which is exactly implied by \eqref{s21} since
$a_1(\xi)=|A^{-1}\xi\cdot \e_\pm|^{1/2}$.
\end{proof}

\medskip

We will need the following frequency-localized form of Lemma \ref{l3},
\begin{corollary}\label{cl3}
For $\phi\in L^{2}(\mathbb{R}^{n})$, $j\in \Z$ and $n\geq 2$, then
\begin{eqnarray}\label{cl31}
\left\|P_{j}Q^{\mathbf{e}_\pm}_{j,10}W_\pm(t)\phi\right\|_{L^{\infty,2}_{\mathbf{e}}}\leq
C2^{-j/2}\|\phi\|_{L^{2}},
\end{eqnarray}
where $\e_\pm$ is defined in \eqref{pme}, and the operator $Q^{\mathbf{e}}_{j,10}$ is defined in
\eqref{q}.
\end{corollary}
\begin{proof}[{Proof of Lemma \ref{cl3}}]
Under the phrase cut-offs $P_{j}Q^{\mathbf{e}_\pm}_{j,10}$, we have the
approximation $D_{\e_\pm}\thicksim 2^{j}$, thus \eqref{cl31} follows directly from \eqref{l31}.
\end{proof}

\medskip

\medskip

Now we give the dyadic maximal function estimate, which generalize Lemma 3.3 in \cite{IK}.

\begin{lemma}[Maximal Function Estimate]\label{l2}
For $\phi\in L^{2}(\mathbb{R}^{n})$, $n\geq2$, $p\ge 2$ and $np\ge 6$, then we have
\begin{eqnarray}\label{l21}
2^{-(\frac{n}{2}-\frac{1}{p})j}\Big\|P_{j}W_\pm(t)\phi\Big\|_{L^{p,\infty}_{\mathbf{e}}}\le C \|\phi\|_{L^{2}},
\end{eqnarray}
where the constant $C$ is independent on $n$, $p$ and $j$.
\end{lemma}

\begin{proof}[{Proof of Lemma \ref{l2}}]From the definition, it is sufficient to show
\[\left\|\int_{\mathbb{R}^{n}}e^{
ix \cdot\xi}e^{it|\xi|^{2}_\pm}\psi_{j}(\xi)\widehat{\phi}(\xi)\,
d\xi\right\|_{L^{p,\infty}_{\mathbf{e}}}\lesssim
2^{(\frac{n}{2}-\frac{1}{p})j}\|\phi\|_{L^{2}}.\]
In view of \eqref{A1},
 it suffices to prove
\begin{eqnarray}\label{l22}
\left\|\int_{\mathbb{R}^{n}}e^{
iA^{-1}x \cdot\xi}e^{it|\xi|^{2}_\pm}\psi_{j}(\xi)\widehat{\phi}(\xi)\,
d\xi\right\|_{L^{p}_{x_{1}}L^{\infty}_{\bar{x},t}}\lesssim
2^{(\frac{n}{2}-\frac{1}{p})j}\|\phi\|_{L^{2}},
\end{eqnarray}
where $\bar{x},\bar{\xi}\in \R^{n-1}$ satisfy $\xi=
(\xi_{1}, \bar{\xi})$, and $x=
(x_{1}, \bar{x})$.
By $A^{-1}=A^t$ and changing of variables, for \eqref{l22} it suffices to show
\begin{eqnarray}\label{l22'}
\left\|\int_{\mathbb{R}^{n}}e^{
ix \cdot\xi}e^{it|A^{-1}\xi|^{2}_\pm}\psi_{j}(\xi)\widehat{\phi}(\xi)\,
d\xi\right\|_{L^{p}_{x_{1}}L^{\infty}_{\bar{x},t}}\lesssim
2^{(\frac{n}{2}-\frac{1}{p})j}\|\phi\|_{L^{2}},
\end{eqnarray}
By standard
$TT^{*}$ argument, for \eqref{l22'} it suffices to show
\begin{eqnarray}\label{l23}
\left\|\int_{\mathbb{R}^{n}}e^{
ix \cdot\xi}e^{it|A^{-1}\xi|^{2}_\pm}\psi_{j}(\xi)
\,d\xi\right\|_{L^{p/2}_{x_{1}}L^{\infty}_{\bar{x},t}}\lesssim
2^{(n-\frac{2}{p})j}.
\end{eqnarray}
Now we begin to prove \eqref{l23}. First we have
\begin{eqnarray}\label{l24}
\Big|\int_{\mathbb{R}^{n}}e^{
ix \cdot\xi}e^{it|A^{-1}\xi|^{2}_\pm}\psi_{j}(\xi) \,d\xi\Big|\lesssim2^{nj}.
\end{eqnarray}
Then by rotation and stationary phase, we have
\begin{eqnarray}\label{l241}
&&\Big|\int_{\mathbb{R}^{n}}e^{
ix \cdot\xi}e^{it|A^{-1}\xi|^{2}_\pm}\psi_{j}(\xi)\,
d\xi\Big|\nonumber\\
&=&\Big|\int_{\mathbb{R}^{n}}e^{
ix\cdot\xi}e^{it|\xi|^{2}_\pm}\psi_{j}(\xi)\, d\xi\Big|
\lesssim|t|^{-\frac{n}{2}}.
\end{eqnarray}
Finally, by integration by parts, for
$|x_{1}|>2^{j+10}|t|$ we have
\begin{eqnarray}\label{l25}
\Big|\int_{\mathbb{R}}e^{
ix \cdot\xi}e^{it|A^{-1}\xi|^{2}_\pm}
\psi_{j}(\xi)d\xi_{1}\Big|\lesssim
\frac{2^{j}}{(1+2^{j}|x_{1}|)^{2}}.
\end{eqnarray}
Let
\[K(x_{1},\bar{x},t)=\int_{\mathbb{R}^{n}}e^{
ix_{1}\xi_{1}}e^{i\bar{x}\bar{\xi}
}e^{it|A^{-1}\xi|^{2}_\pm}\psi_{j}(\xi) \,d\xi_1d\bar{\xi},\]
in view of \eqref{l24},
\eqref{l241} and \eqref{l25}, we have
\begin{eqnarray}\label{ple4}
\sup_{\bar{x},t\in\mathbb{R}}|K(x_{1},\bar{x},t)|\lesssim\left
\{\begin{array}{ll} \displaystyle
2^{nj},\hspace{4.3cm}\text{if}\hspace{0.2cm} |x_1|<2^{-j}
\\
\\
\displaystyle
2^{\frac{nj}{2}}|x_{1}|^{-\frac{n}{2}}+\frac{2^{nj}}{(1+2^{j}|x_{1}|)^{2}},
\hspace{0.3cm}\text{if}\hspace{0.2cm} |x_1|\geq2^{-j}
\end{array}
\right.
\end{eqnarray}
Thus \eqref{l23} follows from \eqref{ple4} since $p\geq2$,
$np\ge 6$.
\end{proof}

\medskip

\begin{lemma}[Strichartz Estimates \cite{KT}]\label{str}
Let $(q, r)$ and $(\tilde{q}, \tilde{r})$   be admissible pairs\footnote{$(q, r)$ is said to be admissible if  $2/q=n(1/2-1/r)$ with $q, r \ge 2$, and $q\neq 2$ for $n=2$.}. We have
\begin{eqnarray}\label{str1}
\left\|W_\pm(t)\phi\right\|_{L_t^qL^r_x}\lesssim
\|\phi\|_{L^{2}},
\end{eqnarray}
\begin{eqnarray}\label{str2}
\Big\|\int_\R W_\pm(-s)F(s)\,ds\Big\|_{L^2_x}\lesssim
\|F\|_{L_t^{\tilde{q}'}L^{\tilde{r}'}_x},
\end{eqnarray}
\begin{eqnarray}\label{str3}
\Big\|\int_0^t
W_\pm(t-s)F(s)\,ds\Big\|_{L_t^qL^r_x}\lesssim
\|F\|_{L_t^{\tilde{q}'}L^{\tilde{r}'}_x},
\end{eqnarray}
where $1/\tilde{q}'+1/\tilde{q}=1$, and
$1/\tilde{r}'+1/\tilde{r}=1$.
\end{lemma}

\medskip

\medskip

\subsection{The main linear estimates}\label{sect:linear}

Now we consider the inhomogeneous IVP
\begin{align}\label{IVP}
 (i\partial_t +\Delta_\pm) u= F(t,x),    \ \  u(x,0)= 0
\end{align}
with $F\in \mathcal{S}(\R\times \R^n)$. Our main result in this
section is

\begin{lemma}[Smoothing effect: inhomogeneous case]\label{le4}The solution of \eqref{IVP} satisfies
\begin{equation}\label{le41}
\big\|D_{\e_\pm} u\big\|_{L^{\infty,2}_{\e}}\leq
C\|F\|_{L^{1,2}_\e}.
\end{equation}
where $D_{\e_\pm}$ is defined by Fourier multiplier $\xi\cdot
\e_\pm$.
\end{lemma}
\begin{proof}[ {Proof of Lemma \ref{le4}}]
Let $u$ satisfy
\begin{align}\label{le411}
\widehat{u}(\tau,\xi)=
c\frac{\widehat{F}(\tau,\xi)}{\tau-|\xi|^2_\pm},
\end{align}
which is a solution of the first equation in \eqref{IVP}.  We have
$$
D_{\e_\pm}u(t,x)=c\int_\R \int_{\R^n} e^{it\tau}e^{ix\cdot\xi}
\frac{\xi\cdot \e_\pm}{\tau-|\xi|^2_\pm}\widehat{F}(\tau,\xi)\,d\xi
\,d\tau.
$$
By definition and \eqref{A1}, we have
\begin{align}\label{le42}
&\big\|D_{\e_\pm} u(t,x)\big\|_{L^{\infty,2}_{\e}}=\big\|D_{\e_\pm} u(t,A^{-1}x)\big\|_{L^{\infty}_{x_1}L^{2}_{\bar{x},t}} \nonumber\\=&c\Big\|\int_\R
\int_{\R^n} e^{it\tau}e^{ix\cdot A\xi}
\frac{\xi\cdot \e_\pm}{\tau-|\xi|^2_\pm}\widehat{F}(\tau,\xi)\,d\xi \,d\tau\Big\|_{L^{\infty}_{x_1}L^{2}_{\bar{x},t}}\nonumber\\
=&c\Big\|\int_\R \int_{\R^n} e^{it\tau}e^{ix\cdot \xi}
\frac{A^{-1}\xi\cdot
\e_\pm}{\tau-|A^{-1}\xi|^2_\pm}\widehat{F}(\tau,A^{-1}\xi)\,d\xi
\,d\tau\Big\|_{L^{\infty}_{x_1}L^{2}_{\bar{x},t}}.
\end{align}
Then we denote $\Omega(\tau,\xi)=\frac{A^{-1}\xi\cdot
\e_\pm}{\tau-|A^{-1}\xi|^2_\pm}$ and apply Plancherel's theorem to
\eqref{le42} in $(\bar{x}, t)$ variables to get
\begin{align}\label{le43}
&\Big\|\int_\R \int_{\R^n} e^{it\tau}e^{ix\cdot \xi}
\Omega(\tau,\xi)\widehat{F}(\tau,A^{-1}\xi)\,d\xi
\,d\tau\Big\|_{L^{2}_{\bar{x},t}}\nonumber\\
=&\Big\|\int_\R e^{ix_1\xi_1}
\Omega(\tau,\xi)\widehat{F}(\tau,A^{-1}\xi)\,d\xi_1 \Big\|_{L^{2}_{\bar{\xi},\tau}}.
\end{align}
Denote
\begin{align}\label{le431}
f(\tau,\xi)=\widehat{F}(\tau,A^{-1}\xi).
\end{align}
Then \eqref{le43} can be rewritten as
\begin{align}\label{le44}
&\Big\|\int_\R e^{ix_1\xi_1}
\Omega(\tau,\xi)f(\tau,\xi)\,d\xi_1
\Big\|_{L^{2}_{\bar{\xi},\tau}}\nonumber\\
=&\Big\|\int_\R
K(\tau,x_1-y_1)\check{f}^{(x_1)}(\tau,y_1,\bar{\xi})\,dy_1
\Big\|_{L^{2}_{\bar{\xi},\tau}},
\end{align}
where $\check{f}^{(x_1)}$ denoting the inverse Fourier transform of $f$ in
$x_1$ variable, and
$$
K(\tau,x_1)=\int_\R e^{ix_1\xi_1} \Omega(\tau,\xi_1,\bar{\xi})\,d\xi_1
$$

Now we claim:
\begin{equation}\label{le45}
K\in L^{\infty}(\R^2),\quad \text{ with norm } M.
\end{equation}

The claim \eqref{le45} combines with Minkowski's inequality and
Plancherel's theorem show that \eqref{le44} can be bounded as follows
\begin{align}\label{le46}
&cM\Big\|\int_\R\check{f}^{(x_1)}(\tau,y_1,\bar{\xi})\,dy_1
\,d\tau\Big\|_{L^{2}_{\bar{\xi},\tau}}\nonumber\\
& \leq cM\int_\R
\|\check{f}^{(x_1)}(\tau,y_1,\bar{\xi})\|_{L^{2}_{\bar{\xi},\tau}}\,dy_1 \\
& = cM\int_\R \|\check{f}(t,y_1,y')\|_{L^{2}_{y',t}}\,dy_1,\nonumber
\end{align}
then apply \eqref{le431} and \eqref{A1}, we continue with
\begin{align}\label{le47}
=& cM\int_\R \|F(t,A^{-1}y)\|_{L^{2}_{y',t}}\,dy_1 \nonumber\\
=& cM \|F(t,y)\|_{L^{1,2}_{\e}}.
\end{align}
which yields \eqref{le41}.

It remains to prove the claim \eqref{le45},
\begin{align}\label{le48}
K(\tau,x_1)=\int_\R e^{ix_1\xi_1}\Omega(\tau,\xi_1,\bar{\xi})\,d\xi_1,
\end{align}
where
\begin{align}\label{le49}
\Omega(\tau,\xi)=\frac{A^{-1}\xi\cdot \e_\pm}{\tau-|A^{-1}\xi|^2_\pm}
\end{align}
if we fix $\bar{\xi}$, $\tau$ and $\e$, and then denote
$E(\xi_1):= A^{-1}\xi\cdot \e_\pm$ and
$F(\xi_1):=\tau-|A^{-1}\xi|^2_\pm$. In view of \eqref{s21}, we
have
$$
\frac{d}{d\xi_1}F(\xi_1)=-2E(\xi_1),
$$
so, we can assume for some $a,b,c \in \R$ depending on $\bar{\xi}$, $\tau$ and $\e$, such that
$$
E(\xi_1)=a\xi_1+b, \quad -F(\xi_1)=\frac{1}{2}a\xi_1^2+b\xi_1+c.
$$
If $a=0$ and $b=0$, then $\Omega(\tau,\xi)=0$ and so $K=0$.

If  $a=0$ and $b\neq0$ then we have
\begin{align}\label{le4110}
-K(\tau,x_1)=\int_\R e^{ix_1\xi_1}\frac{b}{b\xi_1+c}\,d\xi_1,
\end{align}
this is just the Fourier transform of Hilbert transform, thus
bounded.

If  $a\neq0$ and $b\neq0$ then we have
\begin{align}\label{le4111}
-K(\tau,x_1)=&\int_\R
e^{ix_1\xi_1}\frac{a\xi_1+b}{\frac{1}{2}a\xi^2_1+b\xi_1+c}\,d\xi_1\nonumber\\
=&\frac{1}{2}\int_\R
e^{ix_1\xi_1}\frac{\xi_1+\frac{b}{a}}{(\xi_1+\frac{b}{a})^2+c-(\frac{b}{a})^2}\,d\xi_1\nonumber\\
=&\frac{1}{2}e^{ix_1\frac{b}{a}}\int_\R
e^{ix_1\xi_1}\frac{\xi_1}{\xi_1^2+c-(\frac{b}{a})^2}\,d\xi_1
\nonumber\\
=&\frac{1}{2}e^{ix_1\frac{b}{a}}\int_\R
e^{ix_1\xi_1}\frac{\xi_1}{\xi_1^2-\lambda}\,d\xi_1,
\end{align}
which is bounded by a standard argument as in \cite{KPV1} and  we omit the
details.

In general, the $u$ defined in \eqref{le411} may not vanishi at
$t=0$. However by Parseval's indentity we have
\begin{align*}
u(0,x)=&c\int_{\R^n} e^{ix\cdot\xi} \int_\R e^{it\tau}
\frac{1}{\tau-|\xi|^2_\pm}\widehat{F}(\tau,\xi) \,d\tau\,d\xi\\
=&c\int_{\R^n} e^{ix\cdot\xi} \int_\R\widehat{F}^{x}(s,\xi) sgn(s)
e^{-is|\xi|^2_\pm}\,ds \,d\xi\\
=&c \int_\R e^{-is\Delta_\pm} F(s,x) sgn(s)\,ds.
\end{align*}
Now from \eqref{cle31} it follows that $D^{1/2}_{\e_\pm}u(0,x)\in
L^2(\R^n)$, which combine with \eqref{l31} shows that
$$
u(t,x)-e^{-it\Delta_\pm}u(0,x)
$$
is the solution of \eqref{IVP} and satisfies the estimate
\eqref{le41}.
\end{proof}

\medskip

The following result follows directly from Lemma \ref{le4}.

\begin{corollary}\label{cl4}
For $F \in\mathcal{S}(\mathbb{R}^{n+1})$, $j\in \Z$ and $n\geq 2$, then
\begin{eqnarray}\label{cl41}
2^{j/2}\Big\| P_jQ^{\mathbf{e}_\pm}_{j,20}\int_0^tW_\pm(t-s)F(s)\,ds\Big\|_{L^{\infty, 2}_{\e}}\lesssim2^{-j/2}\sup_{\e'\in \s^{n-1} }\|P_jF\|_{L^{1,2}_{\e'}}.
\end{eqnarray}
\end{corollary}

\quad

\begin{lemma}\label{le3}
Let $p\ge 2$ for $n\ge 3$, and $p\ge 3$ for $n= 2$. Then the solutions of
\eqref{IVP} satisfies
\begin{equation}\label{le41a}
2^{-(\frac{n}{2}-\frac{1}{p})j}\|P_j u\|_{L^{p,\infty}_{\e'}}\leq
C2^{-j/2}\sup_{\mathbf{e}\in
\s^{n-1}}||P_jF||_{L_{\mathbf{e}}^{1,2}}.
\end{equation}
where $\e'\in \s^{n-1}$.
\end{lemma}

The case for $\Delta_\pm=\Delta$, $p=2$ and $n\ge 3$ was already proved by
Bejenaru, Ionescu, Kenig, Tataru in \cite{BIKT}. Here we employ a different
argument.

\begin{proof}[{Proof of Lemma \ref{le3}}]

Using a smooth angular partition of unity in frequency, we can assume that $P_ju$ and $P_jF$
is frequency localized to a region $\{\xi:  \xi \cdot \mathbf{e}_\pm\in[2^{j-2},2^{j+2}]\}$
for some $\e \in \mathbb S^{n-1}$, it suffices to prove the stronger bound
  \begin{equation}
    2^{-(\frac{n}{2}-\frac{1}{p})j}\|P_j u\|_{L^{p,\infty}_{\e'}}\lesssim 2^{-j/2} \|F\|_{L^{1,2}_{\e} },
    \label{lud}
  \end{equation}
We rotate the space so that $\mathbf{e} = \mathbf{e}_0$, then the
function \eqref{IVP} reduce to\footnote{Space rotation change the form of the equation, since it's non-elliptic.
For example, in 2-dimension, $i\partial_t u +(\partial_{x_1}^2-\partial_{x_2}^2)u=0$ become to
$i\partial_t u +\partial_{x_1}\partial_{x_2}u=0$ after rotating the space $\pi/4$ clockwise.
And this is the main difficulty of this proof.}
  \begin{equation}\label{IVPb}
    (i\partial_t+\Delta^\e_\pm)u=F\text{ on }\R^n\times\R,\quad u(0)=0.
  \end{equation}
where $\widehat{\Delta^\e_\pm u}(\tau,\xi)=|A^{-1}\xi|_\pm^2\widehat{u}(\tau,\xi)$,
$A$ is defined in \eqref{A} related to $\e$. And here $P_ju$ and $P_jF$
is frequency localized to a region $\{\xi; A^{-1}\xi \cdot \mathbf{e}_\pm
\in[2^{j-2},2^{j+2}]\}$. So for \eqref{lud}, it suffices to prove that the
$u$ in \eqref{IVPb} satisfies
   \begin{equation}
  2^{-(\frac{n}{2}-\frac{1}{p})j}\|P_j u\|_{L^{p,\infty}_{\e'}} \lesssim 2^{-j/2}
   \|F\|_{L^{1}_{x_1}L^2_{\bar{x},t}}.
    \label{lud1}
  \end{equation}
the solution $u$ of \eqref{IVPb} can be expressed as
  \begin{align*}
    u(t,x) = & \int_{\R^{n+1}}\frac{e^{it \tau }e^{ix\cdot\xi}}{\tau-|A^{-1}\xi|^2_\pm}
    \widehat{F}(\tau, \xi) \, d\xi d\tau \\
    = & \int_\R \int_{\R^{n+1}}\frac{e^{it \tau }e^{ix\cdot\xi}}{\tau-|A^{-1}\xi|^2_\pm}
    \Big[\int_{\R^{n}}e^{i\theta \tau }e^{iy\cdot\xi} F(\theta, y_1,y')\,
    dy'd\theta\Big] \, d\xi d\tau\, dy_1 \\
    & = \int_\R u_{y_1}(t,x) d y_1,
  \end{align*}
  where $y=(y_1,y')$, and
  \[
  u_{y_1}(t,x) = \int_{\R^{n+1}}\frac{e^{it \tau }e^{ix\cdot\xi}}{\tau-|A^{-1}\xi|^2_\pm}
    \Big[\int_{\R^{n}}e^{i\theta \tau }e^{iy\cdot\xi} F(\theta, y_1,y')\,
    dy'd\theta\Big] \, d\xi d\tau.
  \]
  For \eqref{lud1}, it suffices to show that
\begin{equation}
   2^{-(\frac{n}{2}-\frac{1}{p})j}\|P_j u_{y_1}\|_{L^{p,\infty}_{\e'}}
   \lesssim 2^{-j/2} \|F(t,y_1,y')\|_{L_{y',t}^2}.
    \label{fin}
\end{equation}
By translation invariance we can set
$y_1=0$ and drop the parameter $y_1$ from the notations. Thus
\[
 u(t,x) =\int_{\R^{n+1}} \frac{e^{it\tau}e^{ix\cdot\xi}}{\tau-|A^{-1}\xi|_\pm^2+i0}
\hat F(\tau,\bar{\xi}) d\xi d\tau.
\]
where $\bar{\xi}\in \R^{n-1}$ and $\xi=(\xi_1,\bar{\xi})$. Now we view $\tau-|A^{-1}\xi|_\pm^2$
as a quadratic of $\xi_1$ variable, then we can decomposition it as
\begin{align}\label{decom:q}
\Theta(\xi_1)=\tau-|A^{-1}\xi|_\pm^2 = c(\xi_1-s_1)(\xi_1-s_2),
\end{align}
where $s_i:=s_i(\tau, \bar{\xi})$. We can assume that $s_1\neq s_2$, since the set $\{(\tau, \bar{\xi}): \ s_1(\tau, \bar{\xi})= s_2(\tau, \bar{\xi})\}$ is a zero-measure set. First, we assume here that $s_1$ and $s_2$ are real numbers. Then we have
\begin{align*}
 u(t,x) =&\int_{\R^{n+1}} \frac{e^{it\tau}e^{ix\cdot\xi}
 }{c(\xi_1-s_1)(\xi_1-s_2)}\hat F(\tau,\bar{\xi})\,d\xi d\tau\\
 =&\int_{\R^{n+1}}\frac{e^{it\tau}e^{ix\cdot\xi}}{c(s_1-s_2)}
 \Big[\frac{1}{\xi_1-s_1}-\frac{1}{\xi_1-s_2}\Big]\hat F(\tau,\bar{\xi})\,d\xi d\tau\\
 =&\int_{\R^{n+1}}\frac{e^{it\tau}e^{ix\cdot\xi}}{c(s_1-s_2)}
 \frac{1}{\xi_1-s_1}\hat F(\tau,\bar{\xi})\,d\xi d\tau\\
 &\qquad-\int_{\R^{n+1}}\frac{e^{it\tau}e^{ix\cdot\xi}}{c(s_1-s_2)}
 \frac{1}{\xi_1-s_2}\hat F(\tau,\bar{\xi})\,d\xi d\tau\\
 := &I_1+I_2.
\end{align*}
By symmetry, we only consider $I_1$. And we continue with
\begin{align*}
I_1 =&\int_{\R^{n+1}}\frac{ e^{it\tau}e^{i\bar{x}\cdot\bar{\xi}}}{c(s_1-s_2)}\Big[\int_\R
  \frac{e^{ix_1\xi_1}}{\xi_1-s_1}\,d\xi_1\Big]\,\hat F(\tau,\bar{\xi})\,d\bar{\xi} d\tau\\
=&\int_{\R^{n+1}}\frac{e^{it\tau}e^{i\bar{x}\cdot\bar{\xi}}}{c(s_1-s_2)}\,[e^{ix_1s_1}i\, sgn(x_1)]\,
\hat F(\tau,\bar{\xi})\,d\bar{\xi} d\tau.
\end{align*}
In view of the definition of $s_1$, we notice that $\tau-|A^{-1}(s_1, \bar{\xi})|_\pm^2=
\Theta(s_1)=0$, thus we have $\tau=|A^{-1}(s_1, \bar{\xi})|_\pm^2$.
Then
\begin{align*}
I_1=&i\, sgn(x_1)\int_{\R^{n+1}}\frac{e^{ix_1s_1}e^{i\bar{x}\cdot\bar{\xi}}}{c(s_1-s_2)}\,
e^{-it|A^{-1}(s_1, \bar{\xi})|_\pm^2}\hat F(\tau,\bar{\xi})\,d\bar{\xi} d\tau.
\end{align*}
Change variable $\eta = s_1 (\tau, \bar{\xi})$ with
\begin{align*}
d\tau=\partial_{\xi_1}\Big|_{\xi_1=s_1}|A^{-1}\xi|_\pm^2d\eta
= \partial_{\xi_1}\Big|_{\xi_1=s_1}\Theta(s_1)\,d\eta
= c(s_1-s_2)\,d\eta.
\end{align*}
where the last step holds since \eqref{decom:q}.
Then we continue with
\begin{align*}
I_1=&i\, sgn(x_1) \int_{\R^{n+1}}e^{ix_1\eta}e^{i\bar{x}\cdot\bar{\xi}}\,
e^{-it|A^{-1}(\eta, \bar{\xi})|_\pm^2}\hat F(|A^{-1}(\eta, \bar{\xi})|_\pm^2,\bar{\xi})\,d\bar{\xi} d\eta\\
=& i\, sgn(x_1) \int_{\R^{n+1}}\,e^{ix\cdot\xi}
e^{-it|A^{-1}\xi|_\pm^2}\hat F(|A^{-1}\xi|_\pm^2,\bar{\xi})\,d\xi\\
:=& i\, sgn(x_1) e^{it\Delta_\pm^\e}v_0.
\end{align*}
where $\hat v_0(\xi) = \hat F(|A^{-1}\xi|_\pm^2,\bar{\xi})$. Then by Lemma \ref{l2}, we have
\[
   2^{-(\frac{n}{2}-\frac{1}{p})j}\|P_j I_1\|_{L^{p,\infty}_{\e'}}
   =2^{-(\frac{n}{2}-\frac{1}{p})j}\|P_je^{it\Delta_\pm^\e}v_0\|_{L^{p,\infty}_{\e'}}
   \lesssim  \|v_0\|_{L^2}.
\]
Thus for \eqref{fin}, it suffices to prove
\begin{align}\label{v0f}
\|v_0\|_{L^2}\lesssim 2^{-j/2} \|F\|_{L_{\bar{x},t}^2},
\end{align}
which follows from changing variable argument in \eqref{s21} and the
frequency localization assumption on $u$.

It remains to consider the case when $s_i$ are complex numbers. Let $s_1=a+ib$ for some $a$, $b\in \R$,
and then we must have $s_2=a-ib$. Thus
\[
\tau+|A^{-1}\xi|_\pm^2 = c(\xi_1-a-ib)(\xi_1-a+ib),
\]
and furthermore
\begin{align}\label{decom}
&\int_\R \frac{e^{ix_1\xi_1}}{\tau+|A^{-1}\xi|_\pm^2}\,d\xi_1\nonumber\\
=&\int_\R \frac{e^{ix_1\xi_1}}{c(\xi_1-a-ib)(\xi_1-a+ib)}\,d\xi_1\nonumber\\
=&\frac{1}{2icb}\int_\R \frac{e^{ix_1\xi_1}}{\xi_1-a-ib}\,d\xi_1
 + \frac{1}{2icb}\int_\R \frac{e^{ix_1\xi_1}}{\xi_1-a+ib}\,d\xi_1\nonumber\\
=&\frac{1}{2icb}e^{ix_1 a}\Big[\int_\R \frac{e^{ix_1\xi_1}}{\xi_1-ib}\,d\xi_1
 + \int_\R \frac{e^{ix_1\xi_1}}{\xi_1+ib}\,d\xi_1\Big].
\end{align}
By the boundness of Hilbert transform, for any $x_1$, $b\in \R$
\begin{align}\label{Hilbert}
\Big|\int_\R \frac{e^{ix_1\xi_1}}{\xi_1+ib}\,d\xi_1\Big|\le C.
\end{align}
Then the  left part of the proof follows from the same argument, where $s_i$ are real,
with \eqref{decom} and \eqref{Hilbert}. Thus we omit the details here.
\end{proof}

\medskip
We notice that the dual version of Lemma \ref{l3} and Corollary \ref{cl3} are given by
\begin{equation}\label{cle31}
\Big\| D_{\e_\pm}^{1/2}\int_\R W_{\pm}(-s)F(s)\,ds\Big\|_{L^2}\leq
C||F||_{L^{1,2}_\e},
\end{equation}
and
\begin{equation}\label{cle32}
\Big\| P_j\int_\R W_{\pm}(-s)Q_{j,10}^{\e_\pm}F(s)\,ds\Big\|_{L^2}\leq
C2^{-j/2}||P_jF||_{L^{1,2}_\e}.
\end{equation}

\begin{lemma}\label{le5}We have the following estimate
\begin{equation}\label{le51}
\Big\| P_j\int_0^tW_\pm(t-s)F(s)\,ds\Big\|_{L^\infty_tL^{2}_{x}}\lesssim2^{-j/2}\sup_{\e\in \s^{n-1} }\|P_jF\|_{L^{1,2}_\e}.
\end{equation}
\end{lemma}
\begin{proof}[{Proof of Lemma \ref{le5}}]
We can assume that $P_jF$ is frequency localized to a region $\{\xi:
\  \xi \cdot \mathbf{e}_\pm\in[2^{j-2},2^{j+2}]\}$ for some
$\mathbf{e} \in \mathbb S^{n-1}$, since finite such regions can
cover the annulus $\{\xi: \  |\xi|\sim 2^{j}\}$. So we may assume
that $P_jF \in L^{1,2}_{\mathbf{e}} $ and it suffices to prove the
stronger bound
\begin{equation}\label{le32}
\Big\|\int_0^tW_\pm(t-s) P_jF(s)\,ds\Big\|_{L^\infty_tL^{2}_{x}}\lesssim2^{-j/2}\|P_jF\|_{L^{1,2}_\e}.
\end{equation}
In view of
\eqref{cle32}, we notice that
\begin{align}\label{le33}
&\Big\| \int W_\pm(t-s)P_jF(s)\,ds\Big\|_{L^{2}_{x}}\nonumber\\ =& \Big\|\int
W_\pm(-s) P_jF(s)\,ds\Big\|_{L^2} \lesssim2^{-j/2}\|P_jF\|_{L^{1,2}_\e},
\end{align}
To conclude we substitute $F(s)$ by $\chi_{[0,t]}(s)F(s)$ then take
the supremum in time in the left hand side of the resulting
inequality.
\end{proof}

\quad

Sometimes, we need Strichartz estimates to deal with the low frequency parts.
\begin{lemma}\label{istr}
Let $(q, r)$ be an admissible pair with $q, r > 2$\footnote{Condition $q,r>2$ is necessary in our argument, since we have used the generalized Christ-Kieslev lemma as in Lemma \ref{A20}.}. We have
\begin{eqnarray}\label{istr1}
\Big\|\int_0^t W_\pm(t-s)P_jF(s)\,ds\Big\|_{L_t^qL^r_x}\lesssim
2^{-j/2}\sup_{\e\in \s^{n-1} }\|P_jf\|_{L^{1,2}_\e},
\end{eqnarray}
and
\begin{eqnarray}\label{istr3}
2^{j/2}\Big\|P_jQ^{\mathbf{e}_\pm}_{j,20}\int_0^t
W_\pm(t-s)F(s)\,ds\Big\|_{L^{\infty, 2}_{\e}}\lesssim
\|P_jF\|_{L_t^{q'}L^{r'}_x}.
\end{eqnarray}
For $p\geq 2$, $np\ge 6$, $\e\in \s^{n-1}$, we have
\begin{eqnarray}\label{istr2}
2^{-(\frac{n}{2}-\frac{1}{p})j}\Big\|P_j\int_0^t
W_\pm(t-s)F(s)\,ds\Big\|_{L^{p,\infty}_{\mathbf{e}}}\lesssim
\|P_jF\|_{L_t^{q'}L^{r'}_x}.
\end{eqnarray}
Finally, the Strichartz estimate
\begin{eqnarray}\label{istr4}
\Big\|P_j\int_0^t
W_\pm(t-s)F(s)\,ds\Big\|_{L^\infty_tL^{2}_{x}}\lesssim
\|P_jF\|_{L_t^{q'}L^{r'}_x},
\end{eqnarray}
where $1/q'+1/q=1$, and
$1/r'+1/r=1$.
\end{lemma}
\begin{proof}[{Proof of Lemma \ref{istr}}] For \eqref{istr1}, using a
smooth angular partition of unity in frequency as in Lemma
\ref{le5}, it suffices to prove
\begin{equation}\label{ile32}
\Big\|\int_0^tW_\pm(t-s) P_jF(s)\,ds\Big\|_{L_t^qL^r_x}\leq
C2^{-j/2}\|P_jF\|_{L^{1,2}_\e}.
\end{equation}
with $P_jf$ is frequency localized to the region $\{\xi;\,\, \xi
\cdot \mathbf{e}_\pm\in[2^{j-2},2^{j+2}]\}$. In view of
\eqref{le32}, it suffices to show \eqref{ile32} for $r>2$. By
\eqref{cle31}, we have
\begin{align}\label{ile33}
&\Big\| \int W_\pm(t-s)P_jF(s)\,ds\Big\|_{L_t^qL^r_x}\nonumber\\ \lesssim& \Big\|\int
W_\pm(-s) P_jF(s)\,ds\Big\|_{L^2} \lesssim2^{-j/2}\|P_jf\|_{L^{1,2}_\e},
\end{align}
since $min\{q,r\}>2$, we can apply Lemma \ref{A20} to
\eqref{ile33} then get \eqref{ile32}.

For \eqref{istr3} and \eqref{istr2}, the proofs are similar and
therefore will be omitted. And \eqref{istr4} follows directly from
Strichartz estimate \eqref{str3}.
\end{proof}

\section{Homogeneous case}\label{sect:homo}

In this section, we will prove the linear and nonlinear estimates for dealing with the homogeneous nonlinearity.

\subsection{Function spaces}\label{sect:space}

 We denote
\begin{align*}
&\|P_jf\|_{M^m_j}:= 2^{-(\frac{n}{2}-\frac{1}{m-1})j}\sup_{\mathbf{e}\in
\s^{n-1}}\|P_jf\|_{L^{m-1,\infty}_{\mathbf{e}}},\\
&\|P_jf\|_{S_j}:=2^{j/2}\sup_{\mathbf{e}\in
\s^{n-1}}\|P_{j}Q^{\mathbf{e}_\pm}_{j,20}f\|_{L^{\infty,2}_{\mathbf{e}}},\\
&\|P_jf\|_{T_j}:=\|P_jf\|_{L^\infty_tL^{2}_{x}},
\end{align*}
and define
\begin{align}\label{y}
\|P_jf\|_{Y^m_j}:=\|P_jf\|_{M^m_j}+\|P_jf\|_{S_j}+\|P_jf\|_{T_j}.
\end{align}
By Corollary \ref{cl3} and Lemma \ref{l2}, for $m\ge 3$, $n(m-1)\ge 6$, we have
\begin{eqnarray}\label{l2l3}
\|P_{j}W_\pm (t)\phi\|_{Y^m_j}\lesssim\|\phi\|_{L^{2}}.
\end{eqnarray}
Denote
\begin{equation}\label{no61}
\|P_ju\|_{N_j}:=2^{-j/2}\sup_{\mathbf{e}\in
\s^{n-1}}||P_ju||_{L_{\mathbf{e}}^{1,2}}.
\end{equation}

Now we are ready to define our working spaces. For $\sigma\geq 0$,
define the resolution space
\begin{equation}\label{no5}
\dot Z_{2,1}^\sigma=\{u\in \mathcal{S}'(\mathbb{R}^{n+1}):\|u\|_{\dot
Z_{2,1}^\sigma}=\sum_{j\in\mathbb{Z}}^\infty 2^{\sigma
j}\|P_ju\|_{Y^m_j}<\infty\},
\end{equation}
and the ``nonlinear space''
\begin{equation}\label{no6}
\dot N_{2,1}^\sigma=\{u\in \mathcal{S}'(\mathbb{R}^{n+1}):\|u\|_{\dot
N_{2,1}^\sigma}=\sum_{j\in\mathbb{Z}}^\infty 2^{\sigma
j}\|P_ju\|_{N_j}<\infty\}.
\end{equation}

\quad

\subsection{Linear estimates}
We now give the following linear estimates
for the  solutions of the non-elliptic Schr\"{o}dinger  equation.

\begin{lemma}\label{le1}
Let $m\ge 3$, $n(m-1)\ge 6$, $\sigma\geq 0$ and $\phi\in \dot B^{\sigma}_{2,1}$, then
$W_\pm (t)\phi\in \dot Z_{2,1}^\sigma$ and
\begin{equation*}
\|W_\pm (t)\phi\|_{\dot Z_{2,1}^\sigma}\leq C\|\phi\|_{\dot
B_{2,1}^\sigma}.
\end{equation*}
\end{lemma}

\begin{proof}  From \eqref{l2l3}, we have
\begin{eqnarray}\label{le12}
\|P_{j}W_\pm (t)\phi\|_{Y^m_j}\leq C\|\tilde{P}_j\phi\|_{L^{2}},
\end{eqnarray}
where $\tilde{P}_j=P_{j-1}+P_j+P_{j+1}$. Then, directly from
\eqref{no5} and \eqref{le12}, we have
\begin{equation*}
\begin{split}
\|W_\pm (t)\phi\|_{\dot Z_{2,1}^\sigma}&=\sum_{j\in\mathbb{Z}}2^{\sigma j}\|P_jW_\pm (t)\phi\|_{Y^m_j}\\
&\leq C\sum_{j\in\mathbb{Z}}2^{\sigma
j}\|\tilde{P}_j\phi\|_{L^{2}}\\
&\leq C\|\phi\|_{\dot
B_{2,1}^\sigma},
\end{split}
\end{equation*}
as desired.
\end{proof}

\begin{lemma}\label{te2}
Let $m\ge 3$, $n(m-1)\ge 6$, $\sigma\geq 0$ and $F\in \dot
N_{2,1}^\sigma$, then $ \int_0^tW(t-s)F(s)\,ds\in \dot
Z_{2,1}^\sigma$, and
\begin{equation}\label{ise}
\Big\| \int_0^tW(t-s)F(s)\,ds\Big\|_{\dot Z_{2,1}^\sigma}\leq
C||F||_{\dot N_{2,1}^\sigma}.
\end{equation}
\end{lemma}

\begin{proof}
By the definition, it is sufficient to show
\begin{align*}
\Big\| P_j\int_0^tW(t-s)F(s)\,ds\Big\|_{Y^m_{j}}\lesssim||P_jF||_{ N_{j}},
\end{align*}
which follows from Corollary \ref{cl4}, Lemma \ref{le3} and Lemma
\ref{le5} since $m\ge 3$ and $n(m-1)\ge 6$.
\end{proof}

\quad

\subsection{Nonlinear estimates for homogeneous nonlinearity}
In this section we estimate the nonlinear term
$
F(u, \bar{u}, \nabla u, \nabla \bar{u})$ in the space $\dot{N}_{2,1}^{s_m}$.

\begin{lemma}\label{ne}
For $m\geq 3$, $(m-1)n\geq 6$, and $s_m=
\frac{n}{2}+\frac{m-2}{m-1}$ we have
\begin{eqnarray}\label{ne1}
\left\|F(\nabla u, \nabla \bar{u})\right\|_{\dot{N}_{2,1}^{s_m}}\leq
C \|u\|_{\dot{Z}_{2,1}^{s_m}}^{m}
\end{eqnarray}
where $F: \mathbb{C}^{2n}\rightarrow \mathbb{C}$ is a homogeneous
polynomial of degree $m$.
\end{lemma}
\begin{proof}[{Proof of Lemma \ref{ne}}]
We can assume that $F(\nabla u, \nabla
\bar{u})=(\partial_{x_1}u)^{m}$. By definition, we have
\begin{eqnarray}\label{st}
\|(\partial_{x_1}u)^{m}\|_{\dot{N}^{s_m}_{2,1}}=\sum_{j\in\mathbb{Z}}2^{(s_m-1/2) j}\sup_{\mathbf{e}\in
\s^{n-1}}\|P_{j}(\partial_{x_1}u)^{m}
\|_{L_{\mathbf{e}}^{1,2}}.
\end{eqnarray}
It is easy to see that
\begin{eqnarray*}
P_{j}(\partial_{x_1}u)^{m}
=\sum_{j_1,\ldots,j_{m}\in
\mathbb{Z}}P_{j}[P_{j_{1}}(\partial_{x_{1}}u)\cdot\ldots\cdot
P_{j_{m}}(\partial_{x_{1}}u)].
\end{eqnarray*}
Furthermore, we have
$$
P_{j}[P_{j_{1}}(\partial_{x_{1}}u)\cdot\ldots\cdot
P_{j_{m}}(\partial_{x_{1}}u)]=0,\text{ unless }\max(j_{1},\ldots,
j_{m})\geq j-C.
$$
Let \[\mathcal{T}^m_j=\{(j_1,\ldots,j_m)\in\mathbb{Z}^{m}: j\leq
\max(j_1,\ldots,j_m)+C\}.\] Then
\begin{eqnarray}\label{el22}
\|P_{j}(\partial_{x_1}u)^{m}
\|_{L_{\mathbf{e}}^{1,2}}\leq
\sum_{(j_1,\ldots,j_{m})\in \mathcal{T}^m_j}\|P_{j}[P_{j_{1}}(\partial_{x_{1}}u)\cdot\ldots\cdot
P_{j_{m}}(\partial_{x_{1}}u)]\|_{L_{\mathbf{e}}^{1,2}}.
\end{eqnarray}
We can assume now $j_{1}=\max(j_1,\ldots,j_{m})$. When $m$ is odd, that is $m=2k+1$ for some $k\in \mathbb{N}$,
we have
\begin{eqnarray*}
&& \|P_{j}[P_{j_{1}}(\partial_{x_{1}}u)\cdot\ldots\cdot
P_{j_{2k+1}}(\partial_{x_{1}}u)]\|_{L_{\mathbf{e}}^{1,2}}\nonumber\\
&\leq&\|P_{j_1}(\partial_{x_1}u)\cdot\ldots\cdot
P_{j_{k+1}}(\partial_{x_1}u)\|_{L^{2}_{t,x}}\\&&\times\|
P_{j_{k+2}}(\partial_{x_1}u)\cdot\ldots\cdot
P_{j_{2k+1}}(\partial_{x_1}u)\|_{L^{2,\infty}_{\mathbf{e}}}\nonumber\\
&\leq&\|P_{j_1}(\partial_{x_1}u)\cdot\ldots\cdot
P_{j_{k+1}}(\partial_{x_1}u)\|_{L^{2}_{t,x}}\prod_{i=k+2}^{2k+1}\|
P_{j_{i}}(\partial_{x_1}u)\|_{L^{2k,\infty}_{\mathbf{e}}}.
\end{eqnarray*}
We can assume
$P_{j_{1}}(\partial_{x_1}u)$ is frequency localized in the region
$\{\xi: \ \xi\cdot \tilde{\mathbf{e}}_\pm \in[2^{j_{1}-2},2^{j_{1}+1}]\}$ for some
  $\tilde{\mathbf{e}}\in
\mathbb{S}^{n-1}$, since finite many such kinds of regions can cover
the annulus $\{\xi: \ |\xi|\sim 2^{j_1}\}$. Then using H\"older's
inequality, we can control the above by
\begin{eqnarray}\label{el23}
&&\|P_{j_1}(\partial_{x_1}u)\|_{L^{\infty,2}_{\tilde{\mathbf{e}}}}\|P_{j_2}(\partial_{x_1}u)\cdot\ldots\cdot
P_{j_{k+1}}(\partial_{x_1}u)\|_{L^{2,\infty}_{\tilde{\mathbf{e}}}}\nonumber\\&&\times\prod_{i=m+2}^{2k+1}\|
P_{j_{i}}(\partial_{x_1}u)\|_{L^{2k,\infty}_{\mathbf{e}}}\nonumber\\
&\leq&\|P_{j_1}(\partial_{x_1}u)\|_{L^{\infty,2}_{\tilde{\mathbf{e}}}}\prod_{i=2}^{k+1}\|
P_{j_{i}}(\partial_{x_1}u)\|_{L^{2k,\infty}_{\tilde{\mathbf{e}}}}\prod_{i=k+2}^{2k+1}\|
P_{j_{i}}(\partial_{x_1}u)\|_{L^{2k,\infty}_{\mathbf{e}}}\nonumber\\
&\leq&C 2^{-j_{1}/2}\|
P_{j_{1}}(\partial_{x_1}u)\|_{S_{j_{1}}}\prod_{i=2}^{2k+1}2^{(\frac{n}{2}-\frac{1}{2k})j_i}\|
P_{j_{i}}(\partial_{x_1}u)\|_{M^{2k+1}_{j_{i}}}\nonumber\\
&\leq& C 2^{j_{1}/2}\|
P_{j_{1}}u\|_{Y^m_{j_{1}}}\prod_{i=2}^{2k+1}2^{(\frac{n}{2}-\frac{1}{m-1}+1)j_i}\|
P_{j_{i}}u\|_{Y^m_{j_{i}}},
\end{eqnarray}
where $m=2k+1$ and $Y^m_j$ norm is defined in \eqref{y}.

When $m$ is even, that is $m=2k$ for some $k\in \N$, then we have
\begin{eqnarray*}
&& \|P_{j}[P_{j_{1}}(\partial_{x_{1}}u) \cdot \ldots \cdot
P_{j_{2k}}(\partial_{x_{1}}u)]\|_{L_{\mathbf{e}}^{1,2}}\nonumber\\
&\leq&\|P_{j_1}(\partial_{x_1}u)\cdot\ldots\cdot
P_{j_{k}}(\partial_{x_1}u)\cdot
|P_{j_{k+1}}(\partial_{x_1}u)|^{1/2}\|_{L^{2}_{t,x}}\\&&\times\||P_{j_{k+1}}(\partial_{x_1}u)|^{1/2}\cdot
P_{j_{k+2}}(\partial_{x_1}u)\cdot\ldots\cdot
P_{j_{2k}}(\partial_{x_1}u)\|_{L^{2,\infty}_{\mathbf{e}}}\nonumber\\
&\leq&\|P_{j_1}(\partial_{x_1}u)\cdot\ldots\cdot
P_{j_{k}}(\partial_{x_1}u)\cdot
|P_{j_{k+1}}(\partial_{x_1}u)|^{1/2}\|_{L^{2}_{t,x}}\\&&\times\|
P_{j_{k+1}}(\partial_{x_1}u)\|^{1/2}_{L^{2k-1,\infty}_{\mathbf{e}}}\prod_{i=k+2}^{2k}\|
P_{j_{i}}(\partial_{x_1}u)\|_{L^{2k-1,\infty}_{\mathbf{e}}}.
\end{eqnarray*}
By the same reason as above, we can assume
$P_{j_{1}}(\partial_{x_1}u)$ is frequency localized in $\{\xi: \ \xi\cdot \tilde{\mathbf{e}}_\pm \in[2^{j_{1}-2},2^{j_{1}+1}]\}$ for some
  $\tilde{\mathbf{e}}\in
\mathbb{S}^{n-1}$.
Using H\"older inequality, we can
control the above by
\begin{eqnarray}\label{el23*}
&&\|P_{j_1}(\partial_{x_1}u)\|_{L^{\infty,2}_{\tilde{\mathbf{e}}}}\|P_{j_2}(\partial_{x_1}u)\cdot\ldots\cdot
P_{j_{k}}(\partial_{x_1}u)\cdot
|P_{j_{k+1}}(\partial_{x_1}u)|^{1/2}\|_{L^{2,\infty}_{\tilde{\mathbf{e}}}}\nonumber\\&&\times\|
P_{j_{k+1}}(\partial_{x_1}u)\|^{1/2}_{L^{2k-1,\infty}_{\mathbf{e}}}\prod_{i=k+2}^{2k}\|
P_{j_{i}}(\partial_{x_1}u)\|_{L^{2k-1,\infty}_{\mathbf{e}}}\nonumber
\end{eqnarray}
By H\"older inequality and definition of $Y^m_j$ norm, we continue with
\begin{eqnarray}\label{el23*}
&\leq&\|P_{j_1}(\partial_{x_1}u)\|_{L^{\infty,2}_{\tilde{\mathbf{e}}}}\prod_{i=2}^{k}\|
P_{j_{i}}(\partial_{x_1}u)\|_{L^{2k-1,\infty}_{\tilde{\mathbf{e}}}}\|
P_{j_{k+1}}(\partial_{x_1}u)\|^{1/2}_{L^{2k-1,\infty}_{\tilde{\e}}}\nonumber\\&&\times\|
P_{j_{k+1}}(\partial_{x_1}u)\|^{1/2}_{L^{2k-1,\infty}_{\mathbf{e}}}\prod_{i=k+2}^{2k}\|
P_{j_{i}}(\partial_{x_1}u)\|_{L^{2k-1,\infty}_{\mathbf{e}}}\nonumber\\
&\leq&C 2^{j_{1}/2}\|
P_{j_{1}}u\|_{Y^m_{j_{1}}}\prod_{i=2}^{2k}2^{(\frac{n}{2}-\frac{1}{m-1}+1)j_i}\|
P_{j_{i}}u\|_{Y^m_{j_{i}}}.
\end{eqnarray}
From \eqref{el22}, \eqref{el23}, and \eqref{el23*} we have
\begin{align}\label{el24}
&  \sup_{\mathbf{e}\in \s^{n-1}}\|P_{j}(\partial_{x_1}u)^{m} \|_{L_{\mathbf{e}}^{1,2}}\nonumber\\
&\lesssim  \sum_{(j_1,\ldots,j_{m})\in \mathcal{T}^m_j} 2^{j_{\max}/2}\|
P_{j_{\max}}u\|_{Y^m_{j_{\max}}}\prod_{j_{i}\neq
j_{\max}}2^{\frac{n}{2}+\frac{m-2}{m-1}}\|
P_{j_{i}}u\|_{Y^m_{j_{i}}}\nonumber\\
&\lesssim \sum_{j_1\geq j-C} 2^{j_{1}/2}\|
P_{j_{1}}u\|_{Y^m_{j_{1}}}\|
u\|_{\dot{Z}_{2,1}^{\frac{n}{2}+\frac{m-2}{m-1}}}^{m-1}.
\end{align}
Then \eqref{ne1} follows from \eqref{st} and \eqref{el24}, since $s_m=\frac{n}{2}+\frac{m-2}{m-1}$ and $F(\nabla u, \nabla
\bar{u})=(\partial_{x_1}u)^{m}$. The proof for general $F$ is similar, thus we omit the details.
\end{proof}

\quad

In order to prove Theorem \ref{t2}, we need the following
estimate.

\begin{lemma}\label{neb}
For $m\geq 3$, $(m-1)n \ge 6$, and $s'_m=\frac{n}{2}-\frac{1}{m-1}$,
we have
\begin{eqnarray}\label{ne1*0}
\left\|u^{m-1} (\lambda\cdot\nabla
u)\right\|_{ \dot{N}_{2,1}^{s'_m}}\lesssim  \|u\|_{
\dot{Z}_{2,1}^{s'_m}}^{m},
\end{eqnarray}
where $\lambda \in \R^n$ is a constant vector.
\end{lemma}

\begin{proof}  The proof  is similar to Lemma \ref{ne}, we only
give the outline. In view of the
proof of Lemma \ref{ne}, it suffices to prove
\begin{eqnarray}\label{el2410}
 \sup_{\mathbf{e}\in \s^{n-1}}\|P_{j}(u^{m-1}( \partial_{x_1}
u)) \|_{L_{\mathbf{e}}^{1,2}}
\lesssim \sum_{j_1\geq j-C} 2^{j_{1}/2}\|
P_{j_{1}}u\|_{Y^m_{j_{1}}}\| u\|_{ \dot{Z}_{2,1}^{s'_m}}^{m-1}.
\end{eqnarray}
Using the notations as in Lemma \ref{ne}, we have
\begin{eqnarray}\label{el22*0}
&& \sup_{\mathbf{e}\in \s^{n-1}}\|P_{j}(u^{m-1}(\partial_{x_1} u))
\|_{L_{\mathbf{e}}^{1,2}}\nonumber\\&=&
\sum_{(j_1,\ldots,j_{m})\in \mathcal{T}^m_j}\sup_{\mathbf{e}\in
\s^{n-1}}\|P_{j}[P_{j_{1}}(\partial_{x_1}u)\cdot P_{j_2} u \cdot\ldots\cdot
P_{j_{m}}u]\|_{L_{\mathbf{e}}^{1,2}}.
\end{eqnarray}
Assume first that $j_1=j_{\max}$, thus $j_1\geq j-C$. In
view of the argument in the proof of Lemma \ref{ne}, we have
\begin{eqnarray}\label{bb0}
&& \|P_{j}[P_{j_{1}}(\partial_{x_1} u)\cdot\ldots\cdot
P_{j_{m}}u]\|_{L_{\mathbf{e}}^{1,2}}\nonumber\\
&\lesssim&  2^{j_1} \sup_{\mathbf{e}\in
\s^{n-1}}\|P_{j_1}Q^{\mathbf{e}_\pm}_{j_1,20}
u\|_{L^{\infty,2}_{\mathbf{e}}}\prod_{i=2}^{m}\sup_{\mathbf{e}\in
\s^{n-1}}\|
P_{j_{i}}u\|_{L^{m-1,\infty}_{\mathbf{e}}}\nonumber\\
&\lesssim&  2^{j_{1}/2}\|
P_{j_{1}}u\|_{Y^m_{j_{1}}}\cdot\prod_{i=2}^{m}2^{s'_{m}j_{i}}\|
P_{j_{i}}u\|_{Y^m_{j_{i}}}.
\end{eqnarray}
Otherwise, we can assume that $j_{2}=j_{\max}$, then we can write
$$
P_{j_{1}}(\partial_{x_1} u)P_{j_{2}}u= [2^{-j_{2}}P_{j_{1}}(\partial_{x_1} u)]2^{j_{2}}
P_{j_{2}}u.
$$
The same argument as before with $\|2^{-j_{2}}P_{j_{1}}(\partial_{x_1}
u)\|_{Y^m_{j_1}}\lesssim\|P_{j_{1}} u\|_{Y^m_{j_1}}$ gives that
\begin{align}\label{bb1}
& \|P_{j}[P_{j_{1}}(\partial_{x_1} u)\cdot P_{j_2}u\cdot\ldots\cdot
P_{j_{m}}u]\|_{L_{\mathbf{e}}^{1,2}}\nonumber\\
\lesssim& \sup_{\mathbf{e}\in
\s^{n-1}}\|2^{-j_{2}}P_{j_{1}}(\partial_{x_1}
u)\|_{L^{m-1,\infty}_{\mathbf{e}}} \cdot 2^{j_{2}}\sup_{\mathbf{e}\in
\s^{n-1}}\| P_{j_{2}}Q^{\mathbf{e}_\pm}_{j_2,20}u
\|_{L^{\infty,2}_{\mathbf{e}}}\nonumber\\
&\qquad\cdot\prod_{i=3}^{m}\sup_{\mathbf{e}\in
\s^{n-1}}\|
P_{j_{i}}u\|_{L^{m-1,\infty}_{\mathbf{e}}}\nonumber\\
\lesssim& 2^{s'_{m}j_{1}}\| 2^{-j_{2}}P_{j_{1}}(\partial_{x_1}
u)\|_{Y^m_{j_{1}}} \cdot 2^{j_{2}/2}\| P_{j_{2}}u\|_{Y^m_{j_{2}}}\cdot\prod_{i=
3}^{m}2^{s'_{m}j_{i}}\| P_{j_{i}}u\|_{Y^m_{j_{i}}}\nonumber\\
\lesssim&  2^{j_{2}/2}\| P_{j_{2}}u\|_{Y^m_{j_{2}}}\cdot\prod_{i\neq
2}^{m}2^{s'_{m}j_{i}}\| P_{j_{i}}u\|_{Y^m_{j_{i}}}.
\end{align}
Then \eqref{el2410} follows from \eqref{el22*0}, \eqref{bb0} and
\eqref{bb1}. Thus we finish the proof.
\end{proof}

\quad

\section{General case}\label{sect:gener}

In this section, we will prove the linear and nonlinear estimates for general DNLS \eqref{ds}.
The main difficulty is the lake of scaling invariance.

\subsection{Function spaces}

In order to prove Theorem \ref{t3}, we introduce the following norms
\begin{align*}
&\|P_ju\|_{\tilde{T}_j}=\|P_ju\|_{L^\infty_t L^{2}_x }+\|P_ju\|_{ L^{2+4/n}_{t,x}},\\
&\|P_ju\|_{\tilde{M}_j}=\sup_{p\in \Z, p\ge 2, np\ge 6}\Big[2^{(\frac{n}{2}-\frac{1}{p})j}\|P_{j}u\|_{L^{p,\infty}_{\mathbf{e}}}\Big].
\end{align*}
Define
\begin{align}\label{y'}
\|P_ju\|_{\tilde{Y}_j}:=\|P_ju\|_{\tilde{M}_j}+\|P_ju\|_{S_j}+\|P_ju\|_{\tilde{T}_j},
\end{align}
 Corollary \ref{cl3}, Lemma \ref{l2} and \ref{str} imply that
\begin{eqnarray}\label{l2l3'}
\|P_{j}W_\pm(t)\phi\|_{\tilde{Y}_j}\leq C\|\phi\|_{L^{2}}.
\end{eqnarray}
Denote
\begin{equation}\label{no61'}
\|u\|_{\tilde{N}_j}:=\inf_{u=v+w}\Big\{||v||_{N_j}+||w||_{L^{\frac{2n+4}{n+4}}}\Big\},
\end{equation}
where the norm $N_j$ is defined in \eqref{no61}.

Now we are ready to define our main spaces. For $\sigma\geq 0$,
define the resolution spaces
\begin{align}\label{no5'}
&\tilde{Z}_{2,1}^\sigma=\{u\in \mathcal{S}(\mathbb{R}^{n+1}):\|u\|_{\tilde{Z}_{2,1}^\sigma}<\infty\},\\
&\|u\|_{\tilde{Z}_{2,1}^\sigma}=\Big(\sum_{j\leq 0}\|P_ju\|^2_{\tilde{Y}_j}\Big)^{1/2}+
\sum_{j\in\mathbb{Z}_+} 2^{\sigma
j}\|P_ju\|_{\tilde{Y}_j},\nonumber
\end{align}
and the ``nonlinear spaces''
\begin{align}\label{no6'}
&\tilde{N}_{2,1}^\sigma=\{u\in \mathcal{S}(\mathbb{R}^{n+1}):\|u\|_{\tilde{N}_{2,1}^\sigma}<\infty\},\\
&\|u\|_{\tilde{Z}_{2,1}^\sigma}=\Big(\sum_{j\leq 0}\|P_ju\|^2_{\tilde{N}_j}\Big)^{1/2}+\sum_{j\in\mathbb{Z}_+} 2^{\sigma
j}\|P_ju\|_{\tilde{N}_j},\nonumber
\end{align}
Our resolution spaces has $l^2$-structure in low frequency part and $l^1$-structure in high frequency part. Since for general nonlinearity, the equation \eqref{ds} has no scaling symmetry, we need measure the low and high frequency parts differently.

\begin{lemma}\label{gc1}
For any admissible  pair  $(q, r)$, if $q\ge 2+4/n$, then
\begin{equation}\label{gc10}
\|u\|_{L^q_t L^r_x}\lesssim\|u\|_{\tilde{Z}_{2,1}^{0}}.
\end{equation}
\end{lemma}
\begin{proof}[{ Proof of Lemma \ref{gc1}}]
Since $q\ge 2+4/n$, by interpolation, we have
\[\|P_{j}u\|_{L^q_t L^r_x}\lesssim \|P_{j}u\|_{\tilde{T}_j},\]
and so
\[\|P_{\ge 1}u\|_{L^q_t L^r_x}\lesssim \sum_{j\ge 1}\|P_{j}u\|_{L^q_t L^r_x}\lesssim \sum_{j\ge 1}\|P_{j}u\|_{\tilde{T}_j}.\]
For low frequency, by Littlewood-Paley square function theorem, we have
\[\|P_{\le 1}u\|_{L^q_t L^r_x}\thicksim \Big\|\Big(\sum_{j\le 1} |P_{j}u|^2\Big)^{1/2}\Big\|_{L^q_t L^r_x}\lesssim \Big(\sum_{j\le 1}\|P_{j}u\|^2_{\tilde{T}_j}\Big)^{1/2}.\]
Thus we finish the proof.
\end{proof}

\quad

\subsection{Linear estimates}

We have the following linear estimate
for the free non-elliptic Schr\"{o}dinger evolution,
\begin{lemma}\label{le1'}
Let $\sigma\geq 0$ and $\phi\in \mathcal{S}(\mathbb{R}^{n})$, then
$W_\pm (t)\phi\in  \tilde{Z}_{2,1}^{\sigma}$ and
\begin{equation*}
\|W_\pm (t)\phi\|_{ \tilde{Z}_{2,1}^{\sigma}}\leq C\|\phi\|_{{B}_{2,1}^{\sigma}}.
\end{equation*}
\end{lemma}
\begin{proof}[Proof of Lemma \ref{le1'}]
This follows from \eqref{l2l3'} and similar argument as in Lemma \ref{le1}.
\end{proof}

\quad

\begin{lemma}\label{te2'}
Let $\sigma\geq 0$ and $F\in \tilde{N}_{2,1}^{\sigma}$ then $
\int_0^tW_\pm (t-s)F(s)\,ds\in \tilde{Z}_{2,1}^{\sigma}$ and
\begin{equation}\label{ise'}
\Big\| \int_0^t W_\pm (t-s)F(s)\,ds\Big\|_{\tilde{Z}_{2,1}^{\sigma}}\leq
C||F||_{\tilde{N}_{2,1}^{\sigma}}.
\end{equation}
\end{lemma}

\begin{proof}[Proof of Lemma \ref{te2'}]
By the definition, it is sufficient to show
\begin{align}\label{te2'1}
\Big\| P_j\int_0^t W_\pm (t-s)F(s)\,ds\Big\|_{\tilde{Y}_{j}}\leq
C||P_j F||_{N_j},
\end{align}
and
\begin{align}\label{te2'2}
\Big\| P_j\int_0^t W_\pm (t-s)F(s)\,ds\Big\|_{\tilde{Y}_{j}}\leq
C||P_j F||_{L_{t,x}^{(2n+4)/(n+4)}}.
\end{align}
Corollary \ref{cl4}, Lemma \ref{le3}, Lemma \ref{le5} and Lemma
\ref{istr} imply \eqref{te2'1}, and \eqref{te2'2} follows from Lemma
\ref{istr}.
\end{proof}

\quad

\subsection{Nonlinear estimate for general nonlinearity}

Since for our solution spaces $X$, $\|u\|_{X}=\|\bar{u}\|_{X}$,
without loss of generality we may assume that
\begin{align*}
 F(u,\bar{u},
\nabla u, \nabla\bar{u}) = F(u, \nabla u):= \sum_{m\le
\kappa+|\nu| <\infty} c_{\kappa \nu} u^\kappa (\nabla u)^{\nu},
\end{align*}
where $(\nabla u)^\nu = u^{\nu_1}_{x_1}...u^{\nu_n}_{x_n} $.
here $m \in \mathbb{N}$, $ m\ge 3$, $n\ge 2$, $\nu\in \mathbb{Z}_+^{n}$.

\begin{lemma}\label{nec}
For $m\geq 3$, $(m-1)n\geq 6$, we have
\begin{align}\label{ne1*}
&\Big\|F(u, \nabla u)\Big\|_{\tilde{N}_{2,1}^{n/2+1}}\lesssim \sum_{m\le
\kappa+|\nu| <\infty}|c_{\kappa \nu}|\cdot\|u\|_{\tilde{Z}_{2,1}^{n/2+1}}^{\kappa+|\nu|}.
\end{align}
\end{lemma}
\begin{proof}[{Proof of Lemma~\ref{nec}}]
Without loss of generality, it suffices to show
\begin{align}\label{nel1}
&\|u^\kappa (\partial_{x_1}u)^{|\nu|}\|_{\tilde{N}_{2,1}^{n/2+1}}\lesssim \|u\|_{\tilde{Z}_{2,1}^{n/2+1}}^{\kappa+|\nu|}.
\end{align}
In view of the definition, it suffices to prove that for $|\nu|\neq0$,
\begin{eqnarray}\label{el241}
 \sup_{\mathbf{e}\in \s^{n-1}}\|P_{j}(u^{\kappa}(\partial_{x_1} u)^{|\nu|}) \|_{L_{\mathbf{e}}^{1,2}}
\leq C\sum_{j\leq j_{1}+C} 2^{j_{1}/2}\|
P_{j_{1}}u\|_{\tilde{Y}_{j_{1}}}\| u\|_{\tilde{Z}_{2,1}^{n/2+1}}^{\kappa+|\nu|-1},
\end{eqnarray}
and for $|\nu|=0$,
\begin{eqnarray}\label{el2411}
\|P_{j}(u^{\kappa}) \|_{L_{t,x}^{(2n+4)/(n+4)}}
\leq C \sum_{j\leq j_{1}+C} \|
P_{j_{1}}u\|_{\tilde{Y}_{j_{1}}}\| u\|_{\tilde{Z}_{2,1}^{n/2+1}}^{\kappa-1}.
\end{eqnarray}

Now we begin to consider \eqref{el241}, let
\begin{eqnarray*}\tilde{s}_{m,j}=
\left \{\begin{array}{ll}
(1+\frac{n}{2}) j \text{ if $j\geq 0$}\\
s'_{m}j \text{ if $j\leq -1$},
\end{array}
\right.
\end{eqnarray*}
where $s'_m=n/2-1/(m-1)$, since $s'_{m}>0$, we have
\[
\sum_{j\in
\mathbb{Z}}2^{\tilde{s}_{m,j}}\|P_ju\|_{\tilde{Y}_j}\lesssim \|u\|_{\tilde{Z}_{2,1}^{\frac{n}{2}+1}}.
\]
Noticing that $s'_m\le s'_{\tilde{m}} < n/2$ for $\tilde{m}\ge m$, in view of the definition we have
\begin{equation}\label{el242}
  \begin{split}
 \|P_{j_{i}}(\partial_{x_1} u)\|_{L^{\tilde{m}-1,\infty}_{\e}}\lesssim
 2^{(s'_{\tilde{m}}+1)j_{i}}\| P_{j_{i}}u\|_{\tilde{M}_{j_{i}}}\lesssim
 2^{\tilde{s}_{m,j_{i}}}\| P_{j_{i}}u\|_{\tilde{Y}_{j_{i}}},
\\
\|P_{j_{i}}(u)\|_{L^{\tilde{m}-1,\infty}_{\e}}\lesssim
 2^{s'_{\tilde{m}} j_{i}}\| P_{j_{i}}u\|_{\tilde{M}_{j_{i}}}\lesssim
 2^{\tilde{s}_{m,j_{i}}}\| P_{j_{i}}u\|_{\tilde{Y}_{j_{i}}},
  \end{split}
\end{equation}
which means that $u$ and $\partial_{x_1} u$ in $L^{\tilde{m}-1,\infty}_{\e}$ have the same
upper bound. Let $\tilde{u}\in \{u, \partial_{x_1} u\}$, thus we have
$\|P_{j_{i}}(\tilde{u})\|_{L^{\tilde{m}-1,\infty}_{\tilde{\e}}}\leq
C2^{\tilde{s}_{m,j_{i}}}\| P_{j_{i}}u\|_{\tilde{Y}_{j_{i}}}$, and
\begin{align}\label{el22*}
&  \sup_{\mathbf{e}\in \s^{n-1}}\|P_{j}(u^{\kappa}(\partial_{x_1} u)^{|\nu|})(s)
\|_{L_{\mathbf{e}}^{1,2}}\nonumber\\
&\leq
\sum_{(j_1,\ldots,j_{\kappa+|\nu|})\in T_{j}^{\kappa+|\nu|}}\sup_{\mathbf{e}\in
\s^{n-1}}\|P_{j}[P_{j_{1}}(\tilde{u})\cdot\ldots\cdot
P_{j_{\kappa+|\nu|}}(\tilde{u})]\|_{L_{\mathbf{e}}^{1,2}}.
\end{align}
By symmetry, we can assume $j_1=j_{\max}$. Furthermore, we can assume
$P_{j_{1}}(\tilde{u})=P_{j_{1}}(\partial_{x_1} u)$, which is the worst case.
In view of the argument in Lemma \ref{ne} and \eqref{el242},
we have
\begin{align}\label{bb}
& \|P_{j}[P_{j_{1}}(\tilde{u})\cdot\ldots\cdot
P_{j_{\kappa+|\nu|}}(\tilde{u})]\|_{L_{\mathbf{e}}^{1,2}}\nonumber\\
&\lesssim  2^{j_1}\sup_{\mathbf{e}\in
\s^{n-1}}\|P_{j_1}Q_{j,20}^{\e_{\pm}}
u\|_{L^{\infty,2}_{\mathbf{e}}}\prod_{i=2}^{\kappa+|\nu|}\sup_{\mathbf{e}\in
\s^{n-1}}\|
P_{j_{i}}\tilde{u}\|_{L^{\kappa+|\nu|-1,\infty}_{\mathbf{e}}}\nonumber\\
&\lesssim  2^{j_{1}/2}\|
P_{j_{1}}u\|_{\tilde{Y}_{j_{1}}}\cdot\prod_{i=2}^{\kappa+|\nu|}2^{\tilde{s}_{m,j_{i}}}\|
P_{j_{i}}u\|_{\tilde{Y}_{j_{i}}}.
\end{align}
Then \eqref{el241} follows from \eqref{el22*} and \eqref{bb}.

Now we turn to \eqref{el2411} and we only consider the case $n\ge 4$, since for $n=2,3$ the proof is similar. By H\"older inequality and \eqref{gc10}, we have
\begin{align}\label{el24111}
& \sum_{(j_1,\ldots,j_{\kappa})\in T_{j}^{\kappa}} \|P_{j}[P_{j_{1}}u\cdot\ldots\cdot
P_{j_{\kappa}}u] \|_{L_{t,x}^{(2n+4)/(n+4)}}\nonumber\\
& \lesssim\sum_{j\le j_1 +C}\sum_{(j_2,\ldots,j_{\kappa})\in \Z^{\kappa-1}}\|P_{j_{1}}u\cdot\ldots\cdot
P_{j_{\kappa}}u \|_{L_{t,x}^{(2n+4)/(n+4)}}\nonumber\\
&\lesssim \sum_{j\le j_1 +C}\|P_{j_{1}}u \|_{L_{t,x}^{2+4/n}}\sum_{(j_2,\ldots,j_{\kappa})\in \Z^{\kappa-1}}\|P_{j_{2}}u\|_{L_{t,x}^q}\prod_{i=3}^{\kappa}\|P_{j_{i}}u\|^{\kappa-2}_{L_{t,x}^\infty}\nonumber\\
& \lesssim \sum_{j\le j_1 +C}\|P_{j_{1}}u \|_{\tilde{Y}_{j_{1}}}\sum_{j_2\in \Z}\|P_{j_{2}}u\|_{L_{t,x}^q}\sum_{(j_3,\ldots,j_{\kappa})\in \Z^{\kappa-2}}\prod_{i=3}^{\kappa}\|P_{j_{i}}u\|^{\kappa-2}_{L_{t,x}^\infty}.
\end{align}
where $q$ satisfies
\[\frac{1}{2+4/n}+\frac{1}{q}=\frac{n+4}{2n+4}.\]
By \eqref{gc10} we have
\[\sum_{j_2\in \Z}\|P_{j_{2}}u\|_{L_{t,x}^{2+4/n}}\lesssim \|u\|_{\tilde{Z}^{0}},\]
and
\begin{align}\label{el24112}
\sum_{j\in \Z}\|P_{j}u\|_{L_{t,x}^\infty}\lesssim \sum_{j\in \Z}2^{nj/2}
\|P_{j}u\|_{L_{t}^\infty L^2_x}\lesssim\|u\|_{\tilde{Z}^{n/2+1}}.
\end{align}
Since $n\ge 4$, so $2+4/n\le q \le \infty$, by interpolation we have
\begin{align}\label{el24113}
\sum_{j_2\in \Z}\|P_{j_{2}}u\|_{L_{t,x}^q}\lesssim \|u\|_{\tilde{Z}^{n/2+1}}.
\end{align}
Thus \eqref{el2411} follows from \eqref{el24111} since \eqref{el24112} and \eqref{el24113}.
\end{proof}

\section{Proof of the main results}\label{sect:proof_main}
\noindent In this section we present the proof of the main results
stated in Section~\ref{sect:intro_main}, and only give the proof for Theorem
\ref{t1}  to demonstrate how our methods
works. We follow the well-known approach via the contraction mapping
principle.

The Cauchy problem \eqref{ds} on the time interval $\R$  is
equivalent to
\begin{equation}\label{ee1}
  \begin{split}
u(t)&=e^{it\Delta_\pm}u_{0}-\int_{0}^{t}e^{i(t-s)\Delta_\pm}F(u,
\bar{u}, \nabla
u, \nabla \bar{u})(s)ds\\
& := e^{it\Delta_\pm}u_{0}-I(u)(t)
  \end{split}
\end{equation}
for regular functions. Whenever we refer to a solution of \eqref{ds}, the
operator equation \eqref{ee1} is assumed to be satisfied.

\begin{proof}[\textbf{Proof of Theorem~\ref{t1}}]
  Lemma \ref{le1} implies that $e^{it
    \Delta_\pm} u_0 \in \dot {Z}_{2,1}^{s_m}$ for $u_0\in
    \dot{B}_{2,1}^{s_m}$ and
  \begin{equation*}
    \|e^{it
    \Delta_\pm} u_0\|_{\dot{Z}_{2,1}^{s_m}}\leq\|u_0\|_{\dot{B}_{2,1}^{s_m}}.
  \end{equation*}
  Let
  \begin{equation*}
    \dot B^1_\delta :=\{u_0 \in \dot{B}_{2,1}^{s_m}(\R^d) \mid \|u_0\|_{\dot
      {B}_{2,1}^{s_m}}<\delta\}
  \end{equation*}
  for $\delta=(4C+4)^{-2}$, with the constant $C>0$ from
  \eqref{ne1}.  Define
  \begin{equation*}
    D_r:=\{u \in  \dot{Z}_{2,1}^{s_m} \mid \|u\|_{\dot{Z}_{2,1}^{s_m}}\leq r\},
  \end{equation*}
  with $r=(4C+4)^{-1}$. Then, for $u_0\in \dot B_\delta$ and $u \in
  D_r$,
  \begin{equation*}
    \|e^{it\Delta_\pm}u_{0}-I(u)(t)\|_{\dot{Z}_{2,1}^{s_m}}\leq
    \delta+Cr^2\leq r,
  \end{equation*}
  due to Lemma \ref{ne}.
  Similarly,
  \begin{align*}
     \left \| I(u)-  I(v)
    \right\|_{\dot{Z}_{2,1}^{s_m}} & \leq C(\|u\|^{m-1}_{\dot{Z}_{2,1}^{s_m}}+\|v\|^{m-1}_{\dot{Z}_{2,1}^{s_m}})
    \|u-v\|_{\dot{Z}_{2,1}^{s_m}}\\
    & \leq \frac12 \|u_1-u_2\|_{\dot{Z}_{2,1}^{s_m}},
  \end{align*}
  so $\Phi : D_r\to D_r, u \mapsto e^{it\Delta_\pm}u_{0}-I(u)(t)$ is a strict contraction. It therefore has a unique
  fixed point in $D_r$, which solves \eqref{ee1}.  By
  implicit function theorem the map $M: \dot B_\delta\to D_r$,
  $u_0\mapsto u$ is analytic because the map $(u_0,u) \mapsto e^{it\Delta_\pm}u_{0}-I(u)(t)$ is analytic.  Due to the embedding
  $\dot{Z}_{2,1}^{s_m}\subset C(\R,
  \dot{B}_{2,1}^{s_m}(\R^d))$, the regularity of the initial data persists
  under the time evolution.

We start to prove the scattering property of system \eqref{dsa} for
small data. For initial data $u_0\in \dot{B}_{2,1}^{s_m}(\R^d)$, $\|u_0\|_{\dot{B}_
  {2,1}^{s_m}}<\delta$, the solution $u$, which was constructed
  above, satisfies
  \begin{eqnarray*}
    u(t)=e^{it\Delta_\pm}\Big(u_{0}-\int_{0}^{t}e^{-is\Delta_\pm}F(\nabla
u, \nabla \bar{u})(s)ds\Big) \ , \ t\in (0,\infty)
  \end{eqnarray*}
   So it is sufficient to prove the existence of the limit
   \begin{eqnarray}\label{lim}
   u_{0}-\int_{0}^{t}e^{-is\Delta_\pm}F(\nabla
u, \nabla \bar{u})(s)ds\to
  u_+\text{ in }\dot{B}_{2,1}^{s_m}(\R^d)\text{ as }t\to \infty\end{eqnarray}
    Without loss of generality we may assume $u \in
  C(\R;\dot{B}_
  {2,1}^{s_m}(\R^d))$ such that $\|u\|_{\dot
    Z^{s_m}_{2,1}}=1$.  Estimate
  \eqref{ne1} implies
  \begin{equation*}
    \sum_j 2^{s_mj}\Big\|e^{it\Delta_\pm} P_j \int_{0}^{t}e^{-is\Delta_\pm}F(\nabla
u, \nabla \bar{u})(s)ds\Big\|_{Y^m_j} \leq C,
  \end{equation*}
Our aim is to show \eqref{lim}, it suffices to show that
\begin{eqnarray}\label{scat}
\lim_{t\to \infty}P_j \int_{0}^{t}e^{-is\Delta_\pm}F(\nabla u,
\nabla \bar{u})(s)ds\in L^2.
\end{eqnarray}
Using the argument in Lemma \ref{ne},
we have for $N>N'$,
\begin{align*}
& \Big\| P_j \int_{N'}^{N}e^{-is\Delta_\pm}F(\nabla
u, \nabla \bar{u})(s)ds\Big\|_{L^{2}_{x}}\\
&\leq  C\sum_{j_1\geq j-C}2^{-j/2}
2^{j_{1}}\|\mathbf{1}_{[N',N]}(t)
P_{j_{1}}Q_{j_1,20}^{\e_\pm}u\|_{L^{\infty,2}_{\e}}\|
u\|_{\dot{Z}_{2,1}^{\frac{d}{2}+\frac{m-2}{m-1}}}^{m-1},
\end{align*}
 in view of the finiteness of $\|P_{j}u\|_{S_j}$,
 so the right hand side of the above goes to zero as
 $N'$ goes to infinity. Thus the convergence \eqref{scat} holds.

  The analyticity of the map
  $V_+:u_0\mapsto u_+$ follows from the analyticity of $M$ shown
  above.
  The existence and analyticity of the local inverse $W_+$ follows
  from the inverse function theorem, because $V_+(0)=0$ and by
  \eqref{ne1} we
  observe $D V_+(0)=Id$.
\end{proof}

The proof for Theorem \ref{t2} and Theorem \ref{t3} are similar,
using Lemma \ref{neb} and Lemma \ref{nec} respectively instead of
Lemma \ref{ne}.

\section{Appendix}

\subsection*{Rotated Christ-Kiselev Lemma}

In this section, we generalize the Christ-Kiselev Lemma
\cite{Ch-Kis,whh}. Denote
\begin{align}
& T f(t)=\int_{-\infty}^{\infty}K(t, t')f(t')dt', \ \ \   T_{re}
f(t)=\int_{0}^{t} K(t, t')f(t')dt'.  \label{A10}
\end{align}
If $T: \ Y_1\to X_1$ implies that  $T_{re}: \ Y_1\to X_1$, then $T:
\ Y_1\to X_1$ is said to be a well restriction operator.

The following lemma from \cite{whh}.

\begin{lemma}\label{A11}
Let $T$ be as in \eqref{A10}. We have the following results.
\begin{itemize}
     \item[\rm (1)] If $\min ( p_1, p_2, p_3) > \max(q_1, q_2, q_3, \ q_1q_3 /q_2)$,
     then $T: L_{x_1}^{q_1}L_{\bar{x}}^{q_2}
L_t^{q_3}(\mathbb{R}^{n+1}) \to L_{x_1}^{p_1}L_{\bar{x}}^{p_2}
L_t^{p_3}(\mathbb{R}^{n+1}) $
     is a well restriction operator.
\item[\rm (2)] If $p_0
> (\vee^3_{i=1} q_i) \vee (q_1q_3/q_2)$, then $T:  L_{x_1}^{q_1}L_{\bar{x}}^{q_2}
L_t^{q_3}(\mathbb{R}^{n+1}) \to L_{t}^{p_0}L_{x}^{p_1} (\mathbb{R}^{n+1})$ is a well restriction operator.
\item[\rm (3)] If $q_0< \min{(p_1, p_2, p_3)}$, then
$T: L_{t}^{q_0}L_{x}^{q_1} (\mathbb{R}^{n+1}) \to
L_{x_1}^{p_1}L_{\bar{x}}^{p_2} L_t^{p_3}(\mathbb{R}^{n+1}) $ is
a well restriction operator.

\end{itemize}
\end{lemma}

By a rotation argument, we can generalize Lemma \ref{A11} to the following.

\begin{lemma}\label{A20}
Let $T$ be as in \eqref{A10}. We have the following results.
\begin{itemize}
\item[\rm (1)] If $ p > 2$,
then $T: L_{\e}^{1,2}(\mathbb{R}^{n+1}) \to L_{\e'}^{p,\infty}(\mathbb{R}^{n+1}) $
is a well restriction operator.

\item[\rm (2)] If $p_0
> 2$, then $T:  L_{\e}^{1,2}(\mathbb{R}^{n+1})  \to L_{t}^{p_0}L^{p_1}_{x}
  (\mathbb{R}^{n+1})$ is a well restriction operator.

\item[\rm (3)] If $q_0< \min{(p_1, p_2)}$, then
$T: L_{t}^{q_0}L_{x}^{q_1} (\mathbb{R}^{n+1}) \to
L_{\e}^{p_1,p_2}(\mathbb{R}^{n+1}) $ is
a well restriction operator.

\end{itemize}
\end{lemma}

\begin{proof}[{Proof of Lemma~\ref{A20}}]
For (1), it suffices to show
\begin{align*}
\Big\| [T_{re}f](x,t)\Big\|_{L_{\e'}^{p,\infty}}\leq \|f(x,t)\|_{L_{\e}^{1,2}},
\end{align*}
under the assumption
\begin{align*}
\Big\| [Tf](x,t)\Big\|_{L_{\e'}^{p,\infty}}\leq \|f(x,t)\|_{L_{\e}^{1,2}}.
\end{align*}
In view of \eqref{A1}, it is sufficient to show
\begin{align}\label{af}
\Big\| [T_{re}f](A'^{-1}x,t)\Big\|_{L_{x_1}^{p}L_{\bar{x},t}^{\infty}}\leq \|f(A^{-1}x,t)\|_{L_{x_1}^{1}L_{\bar{x},t}^{2}},
\end{align}
under the assumption
\begin{align*}
\Big\| [Tf](A'^{-1}x,t)\Big\|_{L_{x_1}^{p}L_{\bar{x},t}^{\infty}}\leq \|f(A^{-1}x,t)\|_{L_{x_1}^{1}L_{\bar{x},t}^{2}},
\end{align*}
if we denote
\begin{align*}
& [\tilde{T} f](x,t)=[T (f(A\cdot))](A'^{-1}x,t), \ \ \   [\tilde{T}_{re} f](x,t)=[T_{re} (f(A\cdot))](A'^{-1}x,t).
\end{align*}
then apply Lemma \ref{A11} (1) to $\tilde{T}$, it follows \eqref{af}.

The proofs for part (2), (3) are similar, thus we omit the details.
\end{proof}

\noindent{\bf Acknowledgment.} Part of the work was carried out
while the second author was visiting the Department of Mathematics
at the University of Chicago under the advising of Prof. Carlos E.
Kenig and the auspices of China Scholarship Council. The authors are
indebted to Prof. Kenig for his valuable suggestions. The second
author also grateful to Prof. Guo for his encouragement and valuable
comments.

\end{document}